\documentclass[]{interact}

\usepackage{epstopdf}
\usepackage[caption=false]{subfig}

\usepackage{amscd,multirow}
\usepackage{graphicx}
\usepackage{hyperref}
\hypersetup{hypertex=true, 
	colorlinks=true, 
	linkcolor=blue, 
	anchorcolor=blue, 
	citecolor=blue}
\usepackage{subeqnarray,cases}

\usepackage[numbers,sort&compress]{natbib}
\bibpunct[, ]{[}{]}{,}{n}{,}{,}
\renewcommand\bibfont{\fontsize{10}{12}\selectfont}
\makeatletter
\def\NAT@def@citea{\def\@citea{\NAT@separator}}
\makeatother

\theoremstyle{plain}
\newtheorem{theorem}{Theorem}[section]
\newtheorem{lemma}[theorem]{Lemma}

\newtheorem{proposition}[theorem]{Proposition}

\theoremstyle{definition}
\newtheorem{definition}[theorem]{Definition}
\newtheorem{example}[theorem]{Example}

\theoremstyle{remark}
\newtheorem{remark}{Remark}

\begin{document}


\title{Strong convergence of an inertial Tikhonov  regularized dynamical system governed by  a maximally comonotone operator}

\author{
\name{Zeng-Zhen Tan\textsuperscript{a}, Rong Hu\textsuperscript{b} and Ya-Ping Fang\textsuperscript{a}}
\affil{\textsuperscript{a} Department of Mathematics, Sichuan University, Chengdu, Sichuan, P.R. China; \textsuperscript{b}Department of Applied Mathematics, Chengdu University of Information Technology, Chengdu, Sichuan, P.R. China}
}

\maketitle

\begin{abstract}
In a Hilbert framework, we consider an inertial Tikhonov regularized dynamical system governed by a maximally comonotone operator, where the damping coefficient is proportional to the square root of the Tikhonov regularization parameter. Under an appropriate setting of the parameters, we  prove  the strong convergence of the trajectory of the proposed system towards the minimum norm element of zeros of the underlying maximally comonotone operator. When the Tikhonov regularization parameter reduces to $\frac{1}{t^q}$ with $0<q<1$, we further establish  some convergence rate results of  the trajectories. Finally, the  validity of the proposed dynamical system is demonstrated by a numerical example.
\end{abstract}


\section{Introduction}
Throughout the paper, $\mathcal{H}$ is a real Hilbert space endowed with the scalar product  $\langle \cdot, \cdot \rangle$ and the induced norm $\|\cdot\|$, respectively. The inclusion problem is to find $x\in \mathcal{H}$ such that
\begin{eqnarray}\label{prob}
0\in \mathcal{A}(x),	
\end{eqnarray} 
where $\mathcal{A}:\mathcal{H}\rightarrow 2^{\mathcal{H}}$ is a point-to-set operator such that the solution set zer$\mathcal{A}:={\mathcal{A}}^{-1}(0)\neq\emptyset$. When the operator $\mathcal{A}$ is the subdifferential of a proper, convex and lower semicontinuous function $f: \mathcal{H}\rightarrow  \mathbb{R}\cup\{+\infty\}$, the inclusion problem $(\ref{prob})$ becomes the optimization problem
\begin{eqnarray}\label{prob2}
\min_{x\in \mathcal{H}} f(x).	
\end{eqnarray} 

A large number of literature investigates the Tikhonov regularized dynamical systems   in order to obtain the strong convergence of the trajectories to the  minimum norm solution of the problem $(\ref{prob2})$. For instance, Attouch and Czarnecki \cite{Attouch2002} studied the dynamical system 
\begin{eqnarray}\label{e1}
\ddot{x}(t)+\gamma\dot{x}(t)+\bigtriangledown f(x(t))+\varepsilon(t)x(t)=0,
\end{eqnarray}
where the damping coefficient $\gamma>0$ is a fixed  constant and $\varepsilon: [0,+\infty)\rightarrow [0,+\infty)$ satisfying $\varepsilon(t)\to 0$ as $t\to +\infty$ is the Tikhonov regularization parameter.
The system \eqref{e1} is a Tikhonov regularized version of the heavy ball with friction system  due to Polyak \cite{Polyak}
$$\ddot{x}(t)+\gamma\dot{x}(t)+\bigtriangledown f(x(t))=0,$$
 and its convergence properties depends upon the speed of  convergence of $\varepsilon(t)$ to zero.   Attouch and Czarnecki \cite{Attouch2002} showed that in the slow parametrization case $\int_{0}^{+\infty}\varepsilon(t) dt=+\infty$, the trajectory generated by $(\ref{e1})$ converges strongly to the  minimum norm solution $x^*$ of the problem $(\ref{prob2})$, i.e., $\lim_{t\to+\infty} \|x(t)-x^*\|\to 0$, and that in the fast  parametrization case $\int_{0}^{+\infty}\varepsilon(t) dt<+\infty$,  the system \eqref{e1} enjoys convergence properties same to  the heavy ball with friction system.
 
In the quest for a faster convergence, Attouch et al. \cite{Attouch2018} considered the following dynamical system with an asymptotically vanishing damping 
 \begin{eqnarray}\label{eS}
 	\ddot{x}(t)+\frac{\alpha}{t}\dot{x}(t)+\bigtriangledown f(x(t))+\varepsilon(t)x(t)=0,
 \end{eqnarray}
which involves an additional  Tikhonov regularization term $\varepsilon(t)x(t)$, compared with the following  known inertial dynamical system introduced by Su et al. \cite{Su2016}
\begin{equation}\label{avd}
\ddot{x}(t)+\frac{\alpha}{t}\dot{x}(t)+\bigtriangledown f(x(t))=0,
\end{equation}
where $\alpha\ge 3$ is a constant. Similar to the system \eqref{e1}, the convergence properties of \eqref{eS}  depend upon the vanishing speed of the parameter $\varepsilon(t)$.  Attouch et al. \cite{Attouch2018} showed that if $\varepsilon(t)$ decreases  rapidly to zero, the system $(\ref{eS})$ owns  fast convergence rates same to \eqref{avd}, and that when $\varepsilon(t)$ tends slowly to zero, the trajectory $x(t)$ converges strongly  in the inferior sense to the  minimum norm solution $x^*$ of the problem $(\ref{prob2})$, i.e., $\liminf_{t\to+\infty}\|x(t)-x^*\|=0$. To obtain the more desired strong convergence result $\lim_{t\to+\infty}\|x(t)-x^*\|=0$,  Attouch et al. \cite{Attouch2018} imposed some additional restictive assumptions on the trajectory $x(t)$.
The analysis method presented in  \cite{Attouch2018} was extended to deal with the Tikhonov regularized dynamical system with an additional explicit Hessian driven damping in \cite{Bot2021} and the Tikhonov regularized dynamical system with an additional implicit Hessian driven damping in \cite{Alecsa2021}. Meanwhile, Xu and Wen \cite{Xu2021} introduced a time scaling parameter into the system $(\ref{eS})$, resulting in the following system
\begin{equation}\label{xw}\ddot{x}(t)+\frac{\alpha}{t}\dot{x}(t)+\beta(t)\big(\bigtriangledown f(x(t))+\varepsilon(t)x(t) \big)=0,
\end{equation}
	where $\beta:[t_0,+\infty)\rightarrow (0,+\infty)$ is a non-negative continuous function, and $t_0>0$. By considering a Hessian driven damping, Zhong et al. \cite{Zhong2024} further proposed the following  system
\begin{equation}\label{zhong}
\ddot{x}(t)+\frac{\alpha}{t}\dot{x}(t)+\beta\frac{d}{dt}\big(\bigtriangledown f(x(t))+\varepsilon(t)x(t) \big)+\bigtriangledown f(x(t))+\varepsilon(t)x(t)=0,
\end{equation}
where $\beta>0$ is a constant.  Some convergergence results  similar to the ones in \cite{Attouch2018} were established in  \cite{Xu2021, Zhong2024}  for \eqref{xw} and \eqref{zhong} respectively. It is worth mentioning that  only the  result  $\liminf_{t\to+\infty}\|x(t)-x^*\|=0$ was  proved for  Tikhonov regularized inertial dynamical systems with the vanishing damp $\frac{\alpha}{t}$, without additional assumptions. Inspired by the convergence properties of the heavy ball with friction method of Polyak in the strongly convex case, Attouch et al.\cite{AttouchL2024}  considered the  Tikhonov regularized inertial dynamical system with the damping coerfficient proportional to the square root of the Tikhonov parameter, as follows
\begin{eqnarray}\label{eH}
    \ddot{x}(t)+\delta \sqrt{\varepsilon(t)}\dot{x}(t)+\bigtriangledown f(x(t))+\varepsilon(t)x(t)=0,
    \end{eqnarray}
where $\delta>0$ is a constant, and proved  the strong convergence in the inferior sense of the trajectory $x(t)$ to the minimizer $x^*$ of the  minimum norm   of $f$.  Attouch et al. \cite{AttouchB2022} further improved  the result $\lim \inf_{t\rightarrow +\infty}\|x(t)-x^*\|=0$ in \cite{AttouchL2024} by
proving $\lim_{t\rightarrow +\infty}\|x(t)-x^*\|=0$, without imposing additional assumptions on the trajectory $x(t)$.  Attouch et al. \cite{AttouchB2023} considered a variant of the system \eqref{eH} by introducing an addtional a Hessian driven damping term, and derived some strong convergence results similar to the ones in \cite{AttouchB2022}. For more results on the  Tikhonov regularized inertial dynamical systems, we refer the reader to \cite{AttouchCh2022,Csetnek2024,AttouchBC2023}.

Second order dynamical systems with vanishing damping can also be used to solve the inclusion problem $(\ref{prob})$. Attouch  and  L\'aszl\'o \cite{AttouchP2019} proposed the following inertial dynamical system
$$\ddot{x}(t)+\frac{\alpha}{t}\dot{x}(t)+\mathcal{A}_{\lambda(t)}(x(t))=0$$
for finding the zero of a maximally monotone operator $\mathcal{A}$, where $\lambda(t)$ is a time-depending parameter, and $\mathcal{A}_{\lambda}$ is the Yosida regularization of $\mathcal{A}$ with a positive regularization parameter $\lambda$.  The analysis has been extended in \cite{AttouchSC2021} by considering a Newton-like correction term
$$\ddot{x}(t)+\frac{\alpha}{t}\dot{x}(t)+\beta\frac{d}{dt}\big(\mathcal{A}_{\lambda(t)}(x(t)) \big)+\mathcal{A}_{\lambda(t)}(x(t))=0.$$
For the nonmonotone case, Tan et al. \cite{Tan2024} proposed a Newton-like inertial dynamical system for solving the inclusion problem $(\ref{prob})$ with $\mathcal{A}$ being a maximally comonotone operator (see Definition \ref{comonotone}), given by
\begin{equation}\label{Tds}
\ddot{x}(t)+\frac{\alpha}{t}\dot{x}(t)+\frac{\beta}{t}\mathcal{A}_{\eta}(x(t))+\frac{d}{dt}\big(\mathcal{A}_{\eta}(x(t))\big)=0,	
\end{equation}
where $\alpha>0$ and $\beta>0$, which has a similar structure of the following inertial dynamical system due to Bo{\c t} et al.\cite{BotC2023}
\begin{equation*}\label{Botds}
\ddot{x}(t)+\frac{\alpha}{t}\dot{x}(t)+\beta(t)\frac{d}{dt}V(x(t))+\frac{1}{2}\big(\dot{\beta}(t)+\frac{\alpha}{t}\beta(t))\big)V(x(t))=0,	
\end{equation*}
where $V:\mathcal{H}\rightarrow \mathcal{H}$ is a monotone and continuous operator and $\beta:[t_0,+\infty)\rightarrow (0,+\infty)$ is a continuously differentiable and nondecreasing function. Tan et al. \cite{Tanar} also considered an implicit Newton-like inertial dynamical system for finding the zero of the when  maximally comonotone operator $\mathcal{A}$. In general, only the weak convergence of the trajectories of the aforementioned dynamical systems can be proved. In order to obtain strong convergence of the trajectories to the   minimum norm solution of  the inclusion problem $(\ref{prob})$, Tikhonov regularization techniques can be used, such as the optimization problem (\ref{prob2}). The Tikhonov regularization traced back to the work of Tikhonov and Arsenine \cite{Tikhonov}.  Tossings \cite{Tossings} extended the Tikhonov regularization to deal with monotone inclusion problems in Hilbert spaces.  Moudafi \cite{Moudafi} proposed a Tikhonov regularization method to find the element of   minimum norm of zeros of a $\gamma-$hypomonotone operator  in Hilbert spaces. When $\mathcal{A}$ is a maximally monotone operator on a Hilbert space,  Cominetti et al. \cite{Cominetti} introduced the following first order dynamical system 
$$-\dot{x}(t)\in \mathcal{A}(x(t))+\varepsilon(t)x(t),$$
and proved its trajectory $x(t)$ converges strongly to the  element of   minimum norm of zeros of  $\mathcal{A}$ provided that  $\varepsilon(t)$ tends to zero as $t\rightarrow +\infty$  and $\int_{0}^{+\infty}\varepsilon(t) dt=+\infty$. Recently, Bo{\c t} et al. \cite{Bot2024} considered the following second order dynamical system
\begin{eqnarray}\label{bot}
\ddot{x}(t)+\frac{\alpha}{t^q}+\beta\frac{d}{dt} \big( \mathcal{A}_{\lambda(t)}(x(t)) \big)+\mathcal{A}_{\lambda(t)}(x(t))+\varepsilon(t)x(t)=0,
\end{eqnarray}
where $\mathcal{A}:\mathcal{H}\rightarrow 2^{\mathcal{H}}$ is a maximally monotone operator, $\alpha>0,\beta\ge 0, 0<q\le 1$ and $\lambda:[t_0,+\infty)\to (0,+\infty)$ is the Yosida parametrization function and $\varepsilon:[t_0,+\infty)\to (0,+\infty)$ is the Tikhonov parametrization function.
Similar to systems \eqref{e1} and  \eqref{eS}, the convergence properties of  \eqref{bot}  depend upon the vanishing speed of the Tikhonov parametrization function $\varepsilon(t)$. Bo{\c t} et al. \cite{Bot2024} proved that the system  \eqref{bot}  exhibits fast convergence rates when  $\varepsilon(t)$ tends rapidly to zero, and that the trajectory of \eqref{bot} converges strongly to the element of  minimum norm of zeros of $\mathcal{A}$. In this paper we consider the following Tikhonov regularized inertial dynamical system governed by a maximally comonotone operator  $\mathcal{A}$
\begin{equation}\label{DS}
\ddot{x}(t)+\delta\sqrt{\epsilon(t)}\dot{x}(t)+\frac{1}{\gamma\sqrt{\epsilon(t)}}\frac{d}{dt} \big(\mathcal{A}_{\eta}x(t)+\epsilon(t)x(t) \big)+\mathcal{A}_{\eta}x(t)+\epsilon(t)x(t)=0
\end{equation}	
where $\delta>0$, $\gamma>0$ and $\epsilon:[t_0, +\infty)\rightarrow (0,+\infty)$ is nonincreasing function, of class $\mathcal{C}^{1}$, such that $\lim_{t\rightarrow +\infty}\epsilon(t)=0$. Under suitable conditions, we shall prove the strong convergence of the trajectory of \eqref{DS} to the minimum norm solution  of  the problem \eqref{prob}.

The paper is organized as follows: Section \ref{S2} presents some basic notation and preliminary results, which we will need in our analysis. In Section \ref{S3}, the global existence and uniqueness result is established for the system $(\ref{DS})$. In Section \ref{S4}, for a general Tikhonov regularization parameter $\epsilon(t)$, we establish the strong convergence of the trajectories to the minimum norm solution of the maximal comonotone inclusion problem. In Section \ref{Pc}, we apply these results to the particular case $\epsilon(t)=\frac{1}{t^q}$. Finally, in Section \ref{Ex} we perform a numerical experiment to illustrate the theoretical results.

\section{Preliminaries}\label{S2}
Given a point-to-set operator $\mathcal{A}: \mathcal{H}\rightarrow 2^{\mathcal{H}}$, it is totally characterized by its graph grap$ \mathcal{A}=\{(x,u)\in \mathcal{H}\times \mathcal{H}: u\in \mathcal{A}x\}$. The domain of $\mathcal{A}$ is the set $\text{ Dom } \mathcal{A}=\{x\in \mathcal{H}: \mathcal{A}x\ne\emptyset\}$.   The inverse of $\mathcal{A}$ is the operator ${\mathcal{A}}^{-1}: \mathcal{H}\rightarrow 2^{\mathcal{H}}$ well-defined through the equivalence $x\in {\mathcal{A}}^{-1}u$ if and only if $u\in \mathcal{A}x$. The set of zeros of $\mathcal{A}$ is the set zer$\mathcal{A}=\{x\in \mathcal{H}:0\in \mathcal{A}x\}$.
Given $\lambda>0$, the resolvent of index $\lambda$ of $\mathcal{A}$ is the operator $J_{\lambda}^{\mathcal{A}}: \mathcal{H}\rightarrow 2^{\mathcal{H}}$ given by 
$$J_{\lambda}^{\mathcal{A}}=(Id+\lambda \mathcal{A})^{-1},$$
and the Yosida regularization of index $\lambda$ of $\mathcal{A}$ is the operator $\mathcal{A}_{\lambda}:\mathcal{H}\rightarrow 2^{\mathcal{H}}$ given by 
\begin{equation}\label{Yosida}
\mathcal{A}_{\lambda}=\frac{1}{\lambda}(Id-J_{\lambda}^{\mathcal{A}}),	
\end{equation}
where $Id$ is the identity operator of $\mathcal{H}$.

Let $B:\mathcal{H}\rightarrow \mathcal{H}$  and let $\theta\in (0,1)$. Recall that (see \cite{BauschkeC}) 
 \begin{itemize}
 \item[(i)] $B$ is $\beta$-Lipschitz continuous for some $\beta>0$ if 
 $$\|Bx-By\|\leq \|x-y\|,~ \forall (x,y)\in \mathcal{H}\times \mathcal{H}.$$ 
When $\beta=1$, $B$ is called nonexpansive.
 \item[(ii)] $B$ is $\theta-$averaged if there exists a nonexpansive operator $N:\mathcal{H}\rightarrow \mathcal{H}$ such that $B=(1-\theta)Id+\theta N$; equivalently, we have 
 $$(1-\theta)\|(Id-B)x-(Id-B)y\|^2\leq \theta(\|x-y\|^2-\|Bx-By\|^2),~\forall (x,y)\in \mathcal{H}\times \mathcal{H}.$$	
 \item[(iii)] $B:\mathcal{H}\rightarrow \mathcal{H}$ is $\beta$-cocoercive for some $\beta>0$ if
$$ \langle Bx-By, x-y\rangle\geq \beta \|Bx-By\|^2, ~\forall (x,y)\in \mathcal{H}\times \mathcal{H}.$$
In this case, $B$ is $\frac{1}{\beta}-$Lipschitz continuous.
 \end{itemize}

\begin{definition} (See \cite[Definition 2.3] {BauschkeMW}) \label{comonotone}
 Let $\mathcal{A}:\mathcal{H}\rightarrow 2^{\mathcal{H}}$ and $\rho\in \mathbb{R}$. Then  
  \begin{itemize}
 \item[(i)] $\mathcal{A}$ is $\rho-$monotone if $\forall (x,u)\in \mbox{gra} \mathcal{A}$, $\forall (y,v)\in \mbox{gra}\mathcal{A}$, we have
 $$\langle x-y,u-v\rangle \geq \rho \|x-y\|^2.$$
 Clearly,  $\mathcal{A}$ is  $\rho-$monotone if and only if $\mathcal{A}-\rho Id$ is monotone. 
\item[(ii)] $\mathcal{A}$ is maximally $\rho-$monotone if $\mathcal{A}$ is $\rho-$monotone and there is no other $\rho-$monotone operator $\mathcal{B}:\mathcal{H}\rightarrow 2^{\mathcal{H}}$ such that $\text{gra} \mathcal{B}$ properly contains $\text{gra}\mathcal{A}$.\\ 
By definition, for a $\rho-$monotone operator $\mathcal{A}$ to be maximally $\rho-$monotone, we'll have to justify that: 
 $$  \langle x-y,u-v\rangle \geq \rho \|x-y\|^2, \forall(y,v)\in \mbox{gra}\mathcal{A}\quad \Rightarrow\quad   (x,u)\in \mbox{gra}\mathcal{A}.$$
 \item[(iii)] $\mathcal{A}$ is $\rho-$comonotone if $\forall (x,u)\in \mbox{gra} \mathcal{A}$, $\forall (y,v)\in \mbox{gra}\mathcal{A}$, 
$$\langle x-y,u-v\rangle \geq \rho \|u-v\|^2.$$
 Then,  $\mathcal{A}$ is $\rho-$comonotone if and only if   ${\mathcal{A}}^{-1}-\rho Id$ is monotone.
\item[(iv)] $\mathcal{A}$ is maximally $\rho-$comonotone if $\mathcal{A}$ is $\rho-$comonotone and there is no other $\rho-$comonotone operator $\mathcal{B}:\mathcal{H}\rightarrow 2^{\mathcal{H}}$ such that gra$\mathcal{B}$ properly contains gra$\mathcal{A}$.\\
Then, for a $\rho-$comonotone operator $\mathcal{A}$ to be maximally $\rho-$comonotone, we'll have to justify that: 
 $$  \langle x-y,u-v\rangle \geq \rho \|u-v\|^2 , \forall(y,v)\in \mbox{gra}\mathcal{A}\quad\Rightarrow \quad (x,u)\in \mbox{gra}\mathcal{A}.$$	
 \end{itemize}
\end{definition}
\begin{remark} (See \cite[Remark 2.4]{BauschkeMW})\label{hypo}
\begin{itemize}
\item[(i)] When $\rho=0$, both $\rho-$monotonicity of $\mathcal{A}$ and $\rho-$comonotonicity of $\mathcal{A}$ reduce to the monotonicity of $\mathcal{A}$; Equivalently to the monotonicity of ${\mathcal{A}}^{-1}$.
\item[(ii)] When $\rho<0$, $\rho-$monotonicity is know as $\rho-$hypomonotonicity, see \cite[Example 12.28] {RockafellarWets} and \cite [Definition 6.9.1] {Burachik}. In this case, the $\rho-$comonotonicity is also known as $\rho-$cohypomonotonicity (see \cite[Definition2.2]{CP}).
\item[(iii)] In passing, we point out that when $\rho>0$, $\rho-$monotonicity of  $\mathcal{A}$ reduces to $\rho-$strong monotonicity of $\mathcal{A}$, while $\rho-$comonotonicity of $\mathcal{A}$ reduces to $\rho-$cocoercivity of $\mathcal{A}$.
\end{itemize}	
\end{remark}

\begin{proposition} (See \cite[Proposition 2.2, Proposition 2.3] {Tan2024})\label{z}
Suppose that $\mathcal{A}:\mathcal{H}\rightarrow 2^{\mathcal{H}}$ is a $\rho-$comonotone operator, $\rho\in \mathbb{R}$ and $\eta>\max\{-2\rho,0\}$. Then  $J_{\eta}^{\mathcal{A}}$ is single-valued, $\text{Dom} J_{\eta}^{\mathcal{A}}=\mathcal{H}$,  and the following conclusions hold:  
\begin{enumerate}
\item[(i)]$\mathcal{A}$ is maximally $\rho-$comonotone $\Leftrightarrow$ ${\mathcal{A}}^{-1}-\rho Id$ is maximally monotone.
\item[(ii)]$\mathcal{A}$ is maximally $\rho-$comonotone $\Leftrightarrow$ $\mathcal{A}_{\eta}$ is $(\rho+\eta)-$cocoercive.
\item[(iii)]$J_{\eta}^{\mathcal{A}}: \mathcal{H}\rightarrow \mathcal{H}$ is $\frac{\eta}{2(\rho+\eta)}-$averaged and ${\mathcal{A}}_{\eta}:  \mathcal{H}\rightarrow \mathcal{H}$ is $\frac{1}{\rho+\eta}-$Lipschitz continuous.
\item[(iv)]$\xi\in J_{\eta}^{\mathcal{A}}x \Leftrightarrow \big(\xi, {\eta}^{-1}(x-\xi)\big)\in \mbox{gra}\mathcal{A}$.
\end{enumerate} 
\end{proposition}

\begin{proposition}(See \cite[Proposition 3]{Tanar})\label{der}
Let $\mathcal{A}:\mathcal{H}\rightarrow 2^{\mathcal{H}}$ be maximally $\rho-$comonotone, $\rho\in \mathbb{R}$, $\eta>\max\{-2\rho,0\}$, and let $x(t)$ be a differentiable function. Then 
$$\bigg \langle \dot{x}(t),~ \frac{d}{dt}\mathcal{A}_{\eta}x(t)\bigg \rangle \geq 0.$$
\end{proposition}

\begin{proposition}\label{solu}
Let $\mathcal{A}:\mathcal{H}\rightarrow 2^{\mathcal{H}}$ be a maximally $\rho-$comonotone operator such that $\mbox{zer} \mathcal{A}\neq\emptyset$. Let $t\longmapsto \epsilon(t)$ be a positive functions defined on $[t_0, +\infty)$ and $\eta>\max\{-2\rho,0\}$. Then $\|x_{\epsilon(t)}\|\leq\|x^*\|$ for all $t\geq t_0$, where $x^*$ is the element of  minimum norm of ${\mbox{zer} \mathcal{A}}(0)$ and $x_{\epsilon(t)}$ denotes the unique zero of the strongly operator $\mathcal{A}_{\eta}+\epsilon(t)Id$, that is, $\mathcal{A}_{\eta}(x_{\varepsilon(t)})+\varepsilon(t)x_{\varepsilon(t)}=0$. Assume further that $\lim_{t\rightarrow +\infty}\epsilon(t)=0$. Then $x_{\epsilon(t)}$ converges strongly to $x^*$ as $t\rightarrow +\infty$.
\end{proposition}

\begin{proof}
According to Proposition \ref{z},  $\mathcal{A}_{\eta}$ is $(\rho+\eta)-$cocoercive. For any $x,y \in \mathcal{H}$, we get 
$$\langle \mathcal{A}_{\eta}x+\epsilon(t)x-\mathcal{A}_{\eta}y-\epsilon(t)y, x-y\rangle=\langle \mathcal{A}_{\eta}x-\mathcal{A}_{\eta}y,x-y\rangle+\epsilon(t)\|x-y\|^2\geq \epsilon(t)\|x-y\|^2.$$	
Hence, $\mathcal{A}_{\eta}+\epsilon(t)Id$ is $\epsilon(t)$-strongly monotone. By \cite[Proposition 2.4]{Tan2024},  the graph of $\mathcal{A}$ is sequentially closed in the weak $\times$ strong topology. The rest of the proof is similar to the one of  \cite[Proposition 5]{Bot2024}, and so we omit it.
\end{proof}

\begin{lemma}\label{ineq}
Let $\mathcal{A}:\mathcal{H}\rightarrow 2^{\mathcal{H}}$ be a maximally $\rho-$comonotone operator, $\epsilon:[t_0, +\infty)\rightarrow [0,+\infty)$ a nonincreasing function of class $\mathcal{C}^{1}$ and $\eta>\max\{-2\rho,0\}$. For every $t\in [t_0,+\infty)$ let $x_{\epsilon(t)}$ be the unique zero of the operator $\mathcal{A}_{\eta}+\epsilon(t)Id$. Then $t\longmapsto x_{\epsilon(t)}$ is almost everywhere differentiable and 
$$\big\|\frac{d}{dt}x_{\epsilon(t)}\big\|\leq-\frac{\dot{\epsilon}(t)}{\epsilon(t)}\big\|x_{\epsilon(t)}\big\| ~~\mbox{for almost all}~~ t\geq t_0.$$
\end{lemma}
\begin{proof}
The differentiability of $x_{\epsilon(t)}$ can refer to  \cite[Lemma 6]{Bot2024}. By Proposition \ref{der},
$$\big\langle \frac{d}{dt}x_{\epsilon(t)}, \frac{d}{dt}\mathcal{A}_{\eta}x_{\epsilon(t)}\big\rangle\geq 0.$$
Since $\mathcal{A}_{\eta}x_{\epsilon(t)}+\epsilon(t)x_{\epsilon(t)}=0$, we get 
$$\big\langle \frac{d}{dt}x_{\epsilon(t)}, -\frac{d}{dt}\left(\epsilon(t)x_{\epsilon(t)}\right)\big\rangle\geq 0.$$
Therefore, 
$$-\dot{\epsilon}(t)\bigg\langle x_{\epsilon(t)},~ \frac{d}{dt}x_{\epsilon(t)}\bigg \rangle \geq \epsilon(t)\big\|\frac{d}{dt}x_{\epsilon(t)}\big\|^2,$$ 
which completes the proof.
\end{proof}

\begin{lemma} (\cite[Proposition 6.2.1]{Haraux})\label{uniq}
Let $\mathbb{X}$ be a Banach space and $f:[t_0,+\infty) \times \mathbb{X}\rightarrow \mathbb{X}$	be a function. Suppose $f$ satisfies the following property:
\begin{itemize}
\item[(i)] $f(t,\cdot):\mathbb{X}\rightarrow\mathbb{X}$ is continuous and 
$$\|f(t,x)-f(t,y)\|\leq M(t,\|x\|+\|y\|)\|x-y\|,~~\forall x,y \in \mathbb{X}$$
for almost all $t\in [t_0,+\infty)$, where $M(t,r)\in L_{loc}^{1}([t_0,+\infty)),\forall r\in (0,+\infty)$;
\item[(ii)] For every $x\in \mathbb{X}$, $f(t,x)\in L_{loc}^{1}([t_0,+\infty))$;
\item[(iii)] $f(t,\cdot): \mathbb{X}\rightarrow \mathbb{X}$ satisfies
$$\|f(t,x)\|\leq P(t)(1+\|x\|)~~\mbox{and}~~P(t)\in L_{loc}^{1}([t_0,+\infty))$$
for almost all $t\in [t_0,+\infty)$.
\end{itemize}
Then, for
$$\frac{d}{dt}x(t)=f(t,x(t)),~~~~x(t_0)=x_0,$$
there exists a unique global trajectory $x: [t_0, +\infty)\rightarrow \mathbb{X}$. 
\end{lemma}

\section{Existence and uniqueness of solutions}\label{S3}

Throughout the paper, we assume that the operator $\mathcal{A}$ and the Tikhonov regularization parameter $\epsilon(t)$ satisfy the following hypotheses:\\
$(H)
\begin{cases}
\mathcal{A}: \mathcal{H}\rightarrow 2^{\mathcal{H}}~\mbox{is a maximally}~\rho\mbox{-comonotone operator with zer}\mathcal{A}\neq \emptyset;\\
\mbox{The parameters satisfy}~ \eta >\max \{ -2\rho, 0\},~ \rho\in \mathbb{R} ~\mbox{and}~\delta>0;\\
\mbox{We denote by}~x^*~\mbox{the element of minimum norm of zer}\mathcal{A};\\
\epsilon:[t_0,+\infty)\rightarrow (0,+\infty)~\mbox{is a nonincreasing function}, ~\lim_{t\rightarrow +\infty}\epsilon(t)=0~\\
 \mbox{and}~ \ddot{\epsilon}(t)~ \mbox{exists}.
\end{cases}$

In this section, we show the existence and uniqueness of strong global solutions to the system $(\ref{DS})$. For the sake of clarity, first we state the definition of a strong global solution.
\begin{definition} 
We say that $x:[t_0,+\infty)\rightarrow \mathcal{H}$ is a strong global solution of the system $(\ref{DS})$ with Cauchy data $(x_0,x_1)\in \mathcal{H}\times \mathcal{H}$ if 
\begin{itemize}
\item [(i)] $x$, $\dot{x}:[t_0,+\infty)\rightarrow \mathcal{H}$ are locally absolutely continuous, in other words, absolutely continuous on each interval $[t_0, T]$ for $t_0<T<+\infty$;
\item [(ii)] $\ddot{x}(t)+\delta\sqrt{\epsilon(t)}\dot{x}(t)+\frac{1}{\gamma\sqrt{\epsilon(t)}}\frac{d}{dt}\big(\mathcal{A}_{\eta}x(t)+\epsilon(t)x(t)\big)+\mathcal{A}_{\eta}x(t)+\epsilon(t)x(t)=0$ for almost every $t\in [t_0,+\infty)$;
\item [(iii)] $x(t_0)=x_0$, $\dot{x}(t_0)=x_1$.
\end{itemize}
\end{definition}

\begin{theorem}\label{theorem1}
Under $(H)$, take $t_0>0$. Then, for any $x_0\in \mathcal{H}$, $x_1 \in \mathcal{H}$, there exists a unique strong global solution $x: [t_0, +\infty) \rightarrow \mathcal{H}$ of the dynamical system $(\ref{DS})$ which satisfies the Cauchy data $x(t_0)=x_0$ and $\dot{x}(t_0)=x_1$.
\end{theorem}
\begin{proof}
First notice that the system $(\ref{DS})$ can be rewritten as

\begin{equation*}
\left\{
\begin{array}{lcl}
\dot{x}(t)=-y(t)-\frac{1}{\gamma \sqrt{\epsilon(t)}}\big(\mathcal{A}_{\eta}x(t)+\epsilon(t)x(t)\big)\\
\dot{y}(t)=-\delta\sqrt{\epsilon(t)}y(t)+\big[\frac{\gamma-\delta}{\gamma}-\frac{1}{\gamma}\frac{d}{dt}(\frac{1}{\sqrt{\epsilon(t)}})\big]\big(\mathcal{A}_{\eta}x(t)+\epsilon(t)x(t)\big).
\end{array}
\right.
\end{equation*}	
Let $Z(t):=\big(x(t),y(t)\big)$ and $F:[t_0,+\infty)\times \mathcal{H}\times\mathcal{H}\rightarrow \mathcal{H}\times\mathcal{H}$ is defined by
$$F(t,x,y):=\bigg(-\frac{1}{\gamma\sqrt{\epsilon(t)}}\big(\mathcal{A}_{\eta}x+\epsilon(t)x\big)-y,-\delta\sqrt{\epsilon(t)}y+[\frac{\gamma-\delta}{\gamma}-\frac{1}{\gamma}\frac{d}{dt}(\frac{1}{\sqrt{\epsilon(t)}})]\big(\mathcal{A}_{\eta}x+\epsilon(t)x)\big)\bigg).$$
Then, the system $(\ref{DS})$  can rewritten as the following first-order dynamical system in the phase space $\mathcal{H}\times\mathcal{H}$ with the Cauchy data $x(t_0)=x_0$ and $y(t_0)=y_0:=-x_1-\frac{1}{\gamma\sqrt{\epsilon(t_0)}}(\mathcal{A}_{\eta}x_0+\epsilon(t_0)x_0)$
\begin{equation*}
\left\{
\begin{array}{lcl}
\dot{Z}(t)=F\big(t,Z(t)\big)\\
Z(t_0)=(x_0,y_0).
\end{array}
\right.
\end{equation*}	

We endow $\mathcal{H}\times \mathcal{H}$ with scalar product $\langle (x,y),(\bar{x},\bar{y})\rangle_{\mathcal{H}\times\mathcal{H}}=\langle x,\bar{x}\rangle+\langle y,\bar{y}\rangle$ and corresponding norm $\|(x,y)\|_{\mathcal{H}\times \mathcal{H}}=\sqrt{\|x\|^2+\|y\|^2}$. Let us write shortly $\kappa(t):=\frac{\gamma-\delta}{\gamma}-\frac{1}{\gamma}\frac{d}{dt}(\frac{1}{\sqrt{\epsilon(t)}})$.

{\it Step 1:} For arbitrary $(x,y),~(\bar{x},\bar{y})\in \mathcal{H}\times \mathcal{H}$, we get
\begin{eqnarray*}
&&\|F(t,x,y)-F(t,\bar{x},\bar{y}\|^2_{\mathcal{H}\times \mathcal{H}}\\
&=&\|\frac{1}{\sqrt{\epsilon(t)}}(\mathcal{A}_{\eta}x+\epsilon(t)x-\mathcal{A}_{\eta}\bar{x}-\epsilon(t)\bar{x})+(y-\bar{y})\|^2\\
&&+\|-\delta\sqrt{\epsilon(t)}(y-\bar{y})+\kappa(t)(\mathcal{A}_{\eta}x+\epsilon(t)x-\mathcal{A}_{\eta}\bar{x}-\epsilon(t)\bar{x})\|^2\\
&\leq&(2+2{\delta}^2\epsilon(t))\|y-\bar{y}\|^2+(\frac{2}{\epsilon(t)}+2{\kappa}^2(t))\|\mathcal{A}_{\eta}x+\epsilon(t)x-\mathcal{A}_{\eta}\bar{x}-\epsilon(t)\bar{x}\|^2\\
&\leq&(2+2{\delta}^2\epsilon(t))\|y-\bar{y}\|^2+(\frac{4}{\epsilon(t)}+4{\kappa}^2(t))\|\mathcal{A}_{\eta}x-\mathcal{A}_{\eta}\bar{x}\|^2\\
&&+(4\epsilon(t)+4{\epsilon}^2(t){\kappa}^2(t))\|x-\bar{x}\|^2\\
&\leq&(2+2{\delta}^2\epsilon(t))\|y-\bar{y}\|^2+\big[(\frac{4}{\epsilon(t)}+4{\kappa}^2(t))(\rho+\eta)^{-2}+4\epsilon(t)+4{\epsilon}^2(t){\kappa}^2(t)\big]\|x-\bar{x}\|^2\\
&\leq&\big[ 2+2{\delta}^2\epsilon(t)+(\frac{4}{\epsilon(t)}+4{\kappa}^2(t))(\rho+\eta)^{-2}+4\epsilon(t)+4{\epsilon}^2(t){\kappa}^2(t)\big]\|(x,y)-(\bar{x},\bar{y})\|^2_{\mathcal{H}\times \mathcal{H}}.
\end{eqnarray*}
Denote $N^2(t):=2+2{\delta}^2\epsilon(t)+(\frac{4}{\epsilon(t)}+4{\kappa}^2(t))(\rho+\eta)^{-2}+4\epsilon(t)+4{\epsilon}^2(t){\kappa}^2(t)$ and $N(t)>0$.
Then,
$$\|F(t,x,y)-F(t,\bar{x},\bar{y})\|_{\mathcal{H}\times \mathcal{H}}\leq N(t)\|(x,y)-(\bar{x},\bar{y})\|_{\mathcal{H}\times \mathcal{H}}.$$
Hence $F(t,\cdot,\cdot)$ is $N(t)-$Lipschitz continuous for every $t\geq t_0$.
Moreover, for any $t\geq t_0$, by the continuity of $\epsilon(t)$ and $\kappa(t)$, we know that $N(\cdot)$ is integrable on $[t_0,T]$ for any $t_0<T<+\infty$. Thus $N(t)\in L_{loc}^{1}([t_0,+\infty)$

{\it Step 2:} For each  $x,y$,  $F(t,x,y)\in L_{loc}^{1}([t_0,+\infty)$ follows from the continuity  of  $\epsilon(t)$ and  the existence of  $\ddot{\epsilon}(t)$.

{\it Step 3:} For arbitrary $z\in \mbox{zer}\mathcal{A}$ and fixed $x,y\in \mathcal{H}$, we have 
\begin{eqnarray*}
&&\|F(t,x,y)\|^2_{\mathcal{H}\times \mathcal{H}}\\
&=& \|y+\frac{1}{\gamma\sqrt{\epsilon(t)}}(\mathcal{A}_{\eta}x+\epsilon(t)x)\|^2+\|-\delta\sqrt{\epsilon(t)}y+\kappa(t)(\mathcal{A}_{\eta}x+\epsilon(t)x)\|^2\\
&=&\|y+\frac{1}{\gamma\sqrt{\epsilon(t)}}(\mathcal{A}_{\eta}x-\mathcal{A}_{\eta}z+\epsilon(t)x)\|^2+\|-\delta\sqrt{\epsilon(t)}y+\kappa(t)(\mathcal{A}_{\eta}x-\mathcal{A}_{\eta}z+\epsilon(t)x)\|^2\\
&\leq &(2+{\delta}^2\epsilon(t))\|y\|^2+(\frac{4}{{\gamma}^2\epsilon(t)}+4{\kappa}^2(t))\big(\|\mathcal{A}_{\eta}x-\mathcal{A}_{\eta}z\|^2+{\epsilon}^2(t)\|x\|^2\big)\\
&\leq & (2+{\delta}^2\epsilon(t))\|y\|^2+(\frac{4}{{\gamma}^2\epsilon(t)}+4{\kappa}^2(t))\frac{1}{(\rho+\eta)^2}\|x-z\|^2+(\frac{4\epsilon(t)}{{\gamma}^2}+4{\epsilon}^2(t){\kappa}^2(t))\|x\|^2\\
&\leq & (2+{\delta}^2\epsilon(t))\|y\|^2+(\frac{8}{{\gamma}^2\epsilon(t)}+8{\kappa}^2(t))\frac{1}{(\rho+\eta)^2}(\|x\|^2+\|z\|^2)\\
&&+(\frac{4\epsilon(t)}{{\gamma}^2}+4{\epsilon}^2(t){\kappa}^2(t))\|x\|^2\\
&\leq& (\frac{8}{{\gamma}^2\epsilon(t)}+8{\kappa}^2(t))\frac{1}{(\rho+\eta)^2}\|z\|^2\\
&&+[(2+{\delta}^2\epsilon(t)+(\frac{8}{{\gamma}^2\epsilon(t)}+8{\kappa}^2(t))\frac{1}{(\rho+\eta)^2}+\frac{4\epsilon(t)}{{\gamma}^2}+4{\epsilon}^2(t){\kappa}^2(t)]\|(x,y)\|^2_{\mathcal{H}\times \mathcal{H}}\\
&\leq& P(t)(1+\|(x,y)\|^2_{\mathcal{H}\times \mathcal{H}}),
\end{eqnarray*}
where 
$$P(t)= (\frac{8}{{\gamma}^2\epsilon(t)}+8{\kappa}^2(t))(1+\frac{1}{(\rho+\eta)^2})\|z\|^2+2+{\delta}^2\epsilon(t)+(\frac{8}{{\gamma}^2\epsilon(t)}+8{\kappa}^2(t))\frac{1}{(\rho+\eta)^2}+\frac{4\epsilon(t)}{{\gamma}^2}$$
$$+4{\epsilon}^2(t){\kappa}^2(t).$$ 
By virtue of the continuity of $\epsilon(t)$, $P(t)\in L_{loc}^{1}([t_0,+\infty)$. 
By the  Cauchy-Lipschitz-Picard Theorem (Lemma \ref{uniq}),  $\dot{Z}(t)=F\big(t,Z(t)\big)$, equivalently, $(\ref{DS})$,  with Cauchy data, admits a unique strong global solution.
\end{proof}

\section{Convergence analysis}\label{S4}

To discuss the convergence properties of the system  $\eqref{DS}$, we consider the energy function $E: [t_0,+\infty) \rightarrow [0, +\infty)$ defined by
\begin{eqnarray}\label{LE}
E(t):=\frac{1}{2}\big\|\gamma\sqrt{\epsilon(t)}(x(t)-x_{\epsilon(t)})+\dot{x}(t)\big\|^2+\langle \mathcal{A}_\eta x(t)+\epsilon(t)x(t), x(t)-x_{\epsilon(t)}\rangle,	
\end{eqnarray}
where $x(t)$ is the trajectory generated by  $\eqref{DS}$  and  $x_{\epsilon(t)}=\big( \mathcal{A}_\eta+\epsilon(t)Id\big)^{-1}(0)$. 

\begin{theorem}\label{Th4.1}
Under $(H)$, let $x:[t_0,+\infty)\rightarrow \mathcal{H}$ be a solution trajectory of the system $(\ref{DS})$ and $E(t)$ be defined in $(\ref{LE})$. Then for any $t\geq t_0$,	
\begin{eqnarray}\label{x2}
\|x(t)-x_{\epsilon(t)}\|^2\leq \frac{E(t)}{\epsilon(t)},
\end{eqnarray}
$$\|\dot{x}(t)\|^2\leq(4+2{\gamma}^2)E(t),$$
and
$$\|\mathcal{A}_\eta x(t)+\epsilon(t)x_{\epsilon(t)}\|^2\leq \frac{E(t)}{\rho+\eta}.$$ 
Therefore, the trajectory $x(t)$ converges strongly to $x^*$ as soon as $\lim_{t\rightarrow +\infty}\frac{E(t)}{\epsilon(t)}=0$, where $x^*$ is the  minimum norm element of $\text{zero}  \mathcal{A}$.
\end{theorem}
\begin{proof}
Since  $\mathcal{A}_\eta+\epsilon(t)Id$ is $\epsilon(t)-$strong monotone,
 $$\langle \mathcal{A}_\eta x(t)+\epsilon(t)x(t), x(t)-x_{\epsilon(t)}\rangle\geq \epsilon(t)\|x(t)-x_{\epsilon(t)}\|^2.$$ 
In virtue of $(\ref{LE})$, we obtain 
$$E(t)\geq\langle \mathcal{A}_\eta x(t)+\epsilon(t)x(t), x(t)-x_{\epsilon(t)}\rangle\geq \epsilon(t)\|x(t)-x_{\epsilon(t)}\|^2,$$
which yields $(\ref{x2})$.

Using again $(\ref{LE})$, we get
$$E(t)\geq \frac{1}{2}\big\|\gamma\sqrt{\epsilon(t)}(x(t)-x_{\epsilon(t)})+\dot{x}(t)\big\|^2.$$
By combining this inequality with $(\ref{x2})$, we have
\begin{eqnarray*}
\|\dot{x}(t)\|^2&=&\|\gamma\sqrt{\epsilon(t)}(x(t)-x_{\epsilon(t)})+\dot{x}(t)-	\gamma\sqrt{\epsilon(t)}(x(t)-x_{\epsilon(t)})\|^2\nonumber\\
&\leq& 2\|\gamma\sqrt{\epsilon(t)}(x(t)-x_{\epsilon(t)})+\dot{x}(t)\|^2+2\|\gamma\sqrt{\epsilon(t)}(x(t)-x_{\epsilon(t)})\|^2\nonumber\\
&\leq&4E(t)+2{\gamma}^2\epsilon(t)\|x(t)-x_{\epsilon(t)}\|^2\nonumber\\
&\leq&(4+2{\gamma}^2)E(t).
\end{eqnarray*} 
It follows from the $(\rho+\eta)-$cocoerciveness of $\mathcal{A}_\eta$  that
\begin{eqnarray}\label{Ac}
&&\big\langle \mathcal{A}_\eta x(t)+\epsilon(t)x(t), x(t)-x_{\epsilon(t)}\big\rangle\nonumber\\
&=&\big\langle \mathcal{A}_\eta x(t)+ \epsilon(t)x(t)-\epsilon(t)x_{\epsilon(t)}+ \epsilon(t)x_{\epsilon(t)}, x(t)-x_{\epsilon(t)}\big\rangle\nonumber\\
&=&\big\langle \mathcal{A}_\eta x(t)-\mathcal{A}_\eta x_{\epsilon(t)},x(t)-x_{\epsilon(t)}\big\rangle+\epsilon(t)\|x(t)-x_{\epsilon(t)}\|^2\nonumber\\
&\geq& (\rho+\eta)\|\mathcal{A}_\eta x(t)+\epsilon(t)x_{\epsilon(t)}\|^2+\epsilon(t)\|x(t)-x_{\epsilon(t)}\|^2\\
&\geq & (\rho+\eta)\|\mathcal{A}_\eta x(t)+\epsilon(t)x_{\epsilon(t)}\|^2\nonumber,
\end{eqnarray}
where the second equality is from $\mathcal{A}_\eta x_{\epsilon(t)}+\epsilon(t)x_{\epsilon(t)}=0$.
Combining $(\ref{LE})$ and \eqref{Ac}, we find
 \begin{eqnarray*}
 E(t)\geq \big\langle \mathcal{A}_\eta x(t)+\epsilon(t)x(t), x(t)-x_{\epsilon(t)}\big\rangle\geq(\rho+\eta)\|\mathcal{A}_\eta x(t)+\epsilon(t)x_{\epsilon(t)}\|^2,
 \end{eqnarray*}
 which yields the desired estimate.
 \end{proof}

\begin{theorem}\label{Th4.2}
Under $(H)$, let $x:[t_0,+\infty)\rightarrow \mathcal{H}$ be a solution trajectory of the system $(\ref{DS})$ and $E(t)$ be defined in $(\ref{LE})$. Let us assume that 
$$\gamma<\delta <\gamma+\frac{\gamma}{\frac{1}{2}{\gamma}^2+1} \text{ and }\lim_{t\rightarrow+\infty}\frac{d}{dt}\big( \frac{1}{\sqrt{\epsilon(t)}}\big)=0.$$ 
Then there exists $t_1\geq t_0$ such that for all $t\geq t_1$, it holds that
$$E(t)\leq \frac{1}{a}\|x^*\|^2\frac{\int_{t_1}^{t}\big[\epsilon^{-\frac{5}{2}}(s){\dot{\epsilon}}^2(s)-\dot{\epsilon}(s)\big]\omega(s)ds}{\omega(t)}+\frac{\omega(t_1)E(t_1)}{\omega(t)},$$
where $a\in (0, 2(2\gamma-\delta)(\rho+\eta))$, $\omega(t)=\exp\big(\int_{t_1}^{t}\mu(s)ds\big )$ and 
$$\mu(t)={\epsilon}^{\frac{1}{2}}(t)[(1-\frac{2}{{\gamma}^2})\frac{d}{dt}(\frac{1}{\sqrt{\epsilon(t)}})+\delta-\gamma].$$	
\end{theorem}

\begin{proof}

We start with computing the derivative of $E(t)$ in two parts. On  one hand, by the classical derivation chain rule and the system $(\ref{DS})$, we arrive at
\begin{eqnarray}\label{yin1}
&&\frac{d}{dt}\frac{1}{2}\big\|\gamma\sqrt{\epsilon(t)}(x(t)-x_{\epsilon(t)})+\dot{x}(t)\big\|^2\nonumber\\
&=&\Big\langle \gamma\sqrt{\epsilon(t)}\big(x(t)-x_{\epsilon(t)}\big)+\dot{x}(t),~\frac{1}{2}\gamma\epsilon^{-\frac{1}{2}}(t)\dot{\epsilon}(t)\big(x(t)-x_{\epsilon(t)}\big)+\gamma\sqrt{\epsilon(t)}\big(\dot{x}(t)-\frac{d}{dt}x_{\epsilon(t)}\big)\nonumber\\
&&-\delta \sqrt{\epsilon(t)}\dot{x}(t)-\frac{1}{\gamma\sqrt{\epsilon(t)}}\frac{d}{dt}\big(\mathcal{A}_{\eta}x(t)+\epsilon(t)x(t)\big)-\big(\mathcal{A}_\eta x(t)+\epsilon(t)x(t)\big )\Big\rangle\nonumber\\
&=& \frac{1}{2}{\gamma}^2\dot{\epsilon}(t)\big\|x(t)-x_{\epsilon(t)}\big\|^2+ \gamma(\gamma-\delta)\epsilon(t)\big\langle x(t)-x_{\epsilon(t)},\dot{x}(t)\big\rangle - {\gamma}^2\epsilon(t)\big\langle x(t)-x_{\epsilon(t)},\frac{d}{dt}x_{\epsilon(t)}\big\rangle\nonumber\\
&& -\big\langle  x(t)-x_{\epsilon(t)}, \frac{d}{dt}\big(\mathcal{A}_\eta x(t)+\epsilon(t)x(t)\big)\big\rangle-\gamma{\epsilon}^{\frac{1}{2}}(t)\big\langle  x(t)-x_{\epsilon(t)},\mathcal{A}_\eta x(t)+ \epsilon(t)x(t)\big\rangle\nonumber\\
&&+\frac{1}{2}\gamma\epsilon^{-\frac{1}{2}}(t)\dot{\epsilon}(t)\big\langle x(t)-x_{\epsilon(t)}, \dot{x}(t)\big\rangle+(\gamma-\delta){\epsilon}^{\frac{1}{2}}(t)\|\dot{x}(t)\|^2-\gamma{\epsilon}^{\frac{1}{2}}(t)\big\langle \dot{x}(t), \frac{d}{dt}x_{\epsilon(t)}\big\rangle\nonumber\\
&&-\frac{1}{\gamma}{\epsilon}^{-\frac{1}{2}}(t)\big\langle \frac{d}{dt}\big(\mathcal{A}_\eta x(t)+\epsilon(t)x(t)\big), \dot{x}(t)\big\rangle -\big\langle \mathcal{A}_\eta x(t)+\epsilon(t)x(t), \dot{x}(t)\big\rangle.
\end{eqnarray}
Notice that the penultimate term of $(\ref{yin1})$ can be estimated as
\begin{eqnarray*}
&&\big\langle \frac{d}{dt}(\mathcal{A}_\eta x(t)+\epsilon(t)x(t)), \dot{x}(t)\big\rangle\nonumber\\
&=&\big\langle \frac{d}{dt}	\mathcal{A}_\eta x(t), \dot{x}(t)\big\rangle+\dot{\epsilon}(t)\big\langle x(t)-x_{\epsilon(t)}, \dot{x}(t)\big\rangle +\dot{\epsilon}(t)\big\langle x_{\epsilon(t)}, \dot{x}(t)\big\rangle+\epsilon(t)\|\dot{x}(t)\|^2\nonumber\\
&\geq & \dot{\epsilon}(t)\big\langle x(t)-x_{\epsilon(t)}, \dot{x}(t)\big\rangle +\dot{\epsilon}(t)\big\langle x_{\epsilon(t)}, \dot{x}(t)\big\rangle+\epsilon(t)\|\dot{x}(t)\|^2,
\end{eqnarray*}
where the inequality is from $\big\langle \frac{d}{dt}	\mathcal{A}_\eta x(t), \dot{x}(t)\big\rangle\geq 0$ in virtue of Proposition \ref{der}. Plugging the above inequality into $(\ref{yin1})$ yields
\begin{eqnarray*}
&&\frac{d}{dt}\frac{1}{2}\big\|\gamma\sqrt{\epsilon(t)}(x(t)-x_{\epsilon(t)})+\dot{x}(t)\big\|^2\\
&\leq &\frac{1}{2}{\gamma}^2\dot{\epsilon}(t)\|x(t)-x_{\epsilon(t)}\|^2+ \epsilon(t)\big[{\gamma}^2-\gamma\delta+(\frac{2}{\gamma}-\gamma)\frac{d}{dt}(\frac{1}{\sqrt{\epsilon(t)}})\big]\big\langle x(t)-x_{\epsilon(t)},\dot{x}(t)\big\rangle\nonumber\\
&& -{\gamma}^2\epsilon(t)\big\langle x(t)-x_{\epsilon(t)},\frac{d}{dt}x_{\epsilon(t)}\big\rangle-\big\langle x(t)-x_{\epsilon(t)}, \frac{d}{dt}\big(\mathcal{A}_\eta x(t)+ \epsilon(t)x(t)\big)\big\rangle\\
&&-\gamma{\epsilon}^{\frac{1}{2}}(t)\big\langle  x(t)-x_{\epsilon(t)},\mathcal{A}_\eta x(t)+ \epsilon(t)x(t)\big\rangle+(\gamma-\delta-\frac{1}{\gamma}){\epsilon}^{\frac{1}{2}}(t)\|\dot{x}(t)\|^2\\
&&-\gamma{\epsilon}^{\frac{1}{2}}(t)\big\langle \dot{x}(t), \frac{d}{dt}x_{\epsilon(t)}\big\rangle-\frac{1}{\gamma}\epsilon^{-\frac{1}{2}}(t)\dot{\epsilon}(t)\big\langle x_{\epsilon(t)}, \dot{x}(t)\big\rangle-\big\langle \dot{x}(t),\mathcal{A}_\eta x(t)+ \epsilon(t)x(t)\big\rangle.
\end{eqnarray*}

On the other hand, let us calculate
\begin{eqnarray*}
&&\frac{d}{dt}\Big\langle \mathcal{A}_\eta x(t)+\epsilon(t)x(t), x(t)-x_{\epsilon(t)}\Big\rangle\\
&=&\big\langle \frac{d}{dt}\big(\mathcal{A}_\eta x(t)+\epsilon(t)x(t)\big), x(t)-x_{\epsilon(t)}\big\rangle+\big\langle \mathcal{A}_\eta x(t)+\epsilon(t)x(t)), \dot{x}(t)\big\rangle\\
&&-\big\langle \mathcal{A}_\eta x(t)+\epsilon(t)x(t),\frac{d}{dt}x_{\epsilon(t)}\big\rangle.
\end{eqnarray*}
Combining the above two relations, we notice that the terms $\big\langle \frac{d}{dt}\big(\mathcal{A}_\eta x(t)+\epsilon(t)x(t)\big), x(t)-x_{\epsilon(t)}\big\rangle$ and $\big\langle \mathcal{A}_\eta x(t)+\epsilon(t)x(t), \dot{x}(t)\big\rangle$ cancel each other out, which gives
\begin{eqnarray}\label{d1}
\frac{d}{dt}E(t)&=&\frac{d}{dt}\frac{1}{2}\big\|\gamma\sqrt{\epsilon(t)}\big(x(t)-x_{\epsilon(t)}\big)+\dot{x}(t)\big\|^2+\frac{d}{dt}\big\langle \mathcal{A}_\eta x(t)+\epsilon(t)x(t), x(t)-x_{\epsilon(t)}\big\rangle\nonumber\\
&\leq&\frac{1}{2}{\gamma}^2\dot{\epsilon}(t)\big\|x(t)-x_{\epsilon(t)}\big\|^2+ \gamma\epsilon(t)\big[\gamma-\delta+(\frac{2}{{\gamma}^2}-1)\frac{d}{dt}(\frac{1}{\sqrt{\epsilon(t)}})\big]\big\langle x(t)-x_{\epsilon(t)},\dot{x}(t)\big\rangle\nonumber\\
&& -{\gamma}^2\epsilon(t)\big\langle x(t)-x_{\epsilon(t)},\frac{d}{dt}x_{\epsilon(t)}\big\rangle-\gamma{\epsilon}^{\frac{1}{2}}(t)\big\langle  x(t)-x_{\epsilon(t)},\mathcal{A}_\eta x(t)+ \epsilon(t)x(t)\big\rangle\nonumber\\
&&+(\gamma-\delta-\frac{1}{\gamma}){\epsilon}^{\frac{1}{2}}(t)\|\dot{x}(t)\|^2-\gamma{\epsilon}^{\frac{1}{2}}(t)\big\langle \dot{x}(t), \frac{d}{dt}x_{\epsilon(t)}\big\rangle-\frac{1}{\gamma}\epsilon^{-\frac{1}{2}}(t)\dot{\epsilon}(t)\big\langle x_{\epsilon(t)}, \dot{x}(t)\big\rangle \nonumber\\
&& -\big\langle \mathcal{A}_\eta x(t)+\epsilon(t)x(t), \frac{d}{dt}x_{\epsilon(t)}\big\rangle.
\end{eqnarray}
Returning to the expansion for $E(t)$ by $(\ref{LE})$, one has 
\begin{eqnarray*}
E(t)&=&\frac{1}{2}{\gamma}^2\epsilon(t)\|x(t)-x_{\epsilon(t)}\|^2+\frac{1}{2}\|\dot{x}(t)\|^2+\gamma{\epsilon}^{\frac{1}{2}}(t)\big\langle x(t)-x_{\epsilon(t)}, \dot{x}(t)\big\rangle\\
&&+\big\langle \mathcal{A}_\eta x(t)+\epsilon(t)x(t), x(t)-x_{\epsilon(t)}\big\rangle.
\end{eqnarray*}
\par
Let us now consider
\begin{eqnarray}\label{ut}
\mu(t)={\epsilon}^{\frac{1}{2}}(t)[(1-\frac{2}{{\gamma}^2})\frac{d}{dt}(\frac{1}{\sqrt{\epsilon(t)}})+\delta-\gamma].	
\end{eqnarray}
Since $\lim_{t\rightarrow +\infty}\frac{d}{dt}(\frac{1}{\sqrt{\epsilon(t)}})=0$ and $\delta>\gamma$, we get $\mu(t)>0$ for all large enough $t$. Multiplying $E(t)$ with $\mu(t)$ we immediately infer that
\begin{eqnarray}\label{d2}
&&\mu(t)E(t)\nonumber\\
&=&{\epsilon}^{\frac{3}{2}}(t)\big[(\frac{{\gamma}^2}{2}-1)\frac{d}{dt}(\frac{1}{\sqrt{\epsilon(t)}})+\frac{1}{2}{\gamma}^2(\delta-\gamma)\big]\|x(t)-x_{\epsilon(t)}\|^2\nonumber\\
&&+{\epsilon}^{\frac{1}{2}}(t)\big[(\frac{1}{2}-\frac{1}{{\gamma}^2})\frac{d}{dt}(\frac{1}{\sqrt{\epsilon(t)}})+\frac{\delta-\gamma}{2}\big]\|\dot{x}(t)\|^2\nonumber\\
&&+\gamma\epsilon(t)\big[(1-\frac{2}{{\gamma}^2})\frac{d}{dt}(\frac{1}{\sqrt{\epsilon(t)}})+\delta-\gamma\big]\big\langle 	x(t)-x_{\epsilon(t)}, \dot{x}(t)\big\rangle\nonumber\\
&&+{\epsilon}^{\frac{1}{2}}(t)\big[(1-\frac{2}{{\gamma}^2})\frac{d}{dt}(\frac{1}{\sqrt{\epsilon(t)}})+\delta-\gamma\big]\big\langle \mathcal{A}_\eta x(t)+\epsilon(t)x(t), x(t)-x_{\epsilon(t)}\big\rangle.
\end{eqnarray}
Further, by adding $(\ref{d1})$ and $(\ref{d2})$, we derive
\begin{eqnarray}\label{eq1}
&&\frac{d}{dt}E(t)+\mu(t)E(t)\nonumber\\
&\leq&{\epsilon}^{\frac{3}{2}}(t)\big[(-1-\frac{{\gamma}^2}{2})\frac{d}{dt}(\frac{1}{\sqrt{\epsilon(t)}})+\frac{1}{2}{\gamma}^2(\delta-\gamma)\big]\|x(t)-x_{\epsilon(t)}\|^2\nonumber\\
&&+{\epsilon}^{\frac{1}{2}}(t)\big[(\frac{1}{2}-\frac{1}{{\gamma}^2})\frac{d}{dt}(\frac{1}{\sqrt{\epsilon(t)}})+\frac{\gamma-\delta}{2}-\frac{1}{\gamma}\big]\|\dot{x}(t)\|^2\nonumber\\
&&+{\epsilon}^{\frac{1}{2}}(t)\big[(1-\frac{2}{{\gamma}^2})\frac{d}{dt}(\frac{1}{\sqrt{\epsilon(t)}})+\delta-2\gamma\big]\big\langle \mathcal{A}_\eta x(t)+\epsilon(t)x(t), x(t)-x_{\epsilon(t)}\big\rangle\nonumber\\
&&-{\gamma}^2\epsilon(t)\big\langle x(t)-x_{\epsilon(t)}, \frac{d}{dt}x_{\epsilon(t)}\big\rangle-\gamma\epsilon^{\frac{1}{2}}(t)\big\langle \dot{x}(t), \frac{d}{dt}x_{\epsilon(t)}\big\rangle\nonumber\\
&&-\frac{1}{\gamma}\epsilon^{-\frac{1}{2}}(t)\dot{\epsilon}(t)\big\langle x_{\epsilon(t)}, \dot{x}(t)\rangle-\big\langle \mathcal{A}_\eta x(t)+\epsilon(t)x(t), \frac{d}{dt}x_{\epsilon(t)}\big\rangle.
\end{eqnarray}

Next, we will estimate the inner product terms in $(\ref{eq1})$. Since the assumptions\\
 $\lim_{t\rightarrow +\infty}\frac{d}{dt}\big(\frac{1}{\sqrt{\epsilon(t)}} \big)=0$ and $\delta<\gamma+\frac{\gamma}{\frac{1}{2}{\gamma}^2+1}$, we immediately deduce that 
$$(1-\frac{2}{{\gamma}^2})\frac{d}{dt}(\frac{1}{\sqrt{\epsilon(t)}})+\delta-2\gamma< 0$$ 
for all large enough $t$. It follows from  \eqref{Ac} that
\begin{eqnarray*}
&&{\epsilon}^{\frac{1}{2}}(t)\big[(1-\frac{2}{{\gamma}^2})\frac{d}{dt}(\frac{1}{\sqrt{\epsilon(t)}})+\delta-2\gamma\big]\langle \mathcal{A}_\eta x(t)+\epsilon(t)x(t), x(t)-x_{\epsilon(t)}\big\rangle\\
&\leq &{\epsilon}^{\frac{1}{2}}(t)\big[(1-\frac{2}{{\gamma}^2})\frac{d}{dt}(\frac{1}{\sqrt{\epsilon(t)}})+\delta-2\gamma\big](\rho+\eta)\|\mathcal{A}_\eta x(t)+\epsilon(t)x_{\epsilon(t)}\|^2\\
&&+{\epsilon(t)}^{\frac{3}{2}}\big[(1-\frac{2}{{\gamma}^2})\frac{d}{dt}(\frac{1}{\sqrt{\epsilon(t)}})+\delta-2\gamma\big]\|x(t)-x_{\epsilon(t)}\|^2.
\end{eqnarray*}

On the other hand, using the Cauchy-Schwarz inequality we get for $\forall ~a,~b,~c>0$,
$$-{\gamma}^2\epsilon(t) \big\langle x(t)-x_{\epsilon(t)}, \frac{d}{dt}x_{\epsilon(t)}\big\rangle\leq \frac{{\gamma}^2}{2}\epsilon(t)\Big(b{\epsilon}^{\frac{1}{2}}(t)\| x(t)-x_{\epsilon(t)}\|^2+\frac{1}{b}{\epsilon}^{-\frac{1}{2}}(t)\|\frac{d}{dt}x_{\epsilon(t)}\|^2\Big),$$
$$-\gamma{\epsilon}^{\frac{1}{2}}(t)\big\langle \dot{x}(t), \frac{d}{dt}x_{\epsilon(t)}\big\rangle\leq \frac{\gamma{\epsilon}^{\frac{1}{2}}(t)}{2}\Big(\frac{2}{{\gamma}^2}\|\dot{x}(t)\|^2+ \frac{{\gamma}^2}{2}\|\frac{d}{dt}x_{\epsilon(t)}\|^2\Big),$$
$$-\frac{1}{\gamma}\epsilon^{-\frac{1}{2}}(t)\dot{\epsilon}(t)\big \langle  x_{\epsilon(t)}, \dot{x}(t)\big\rangle\leq -\frac{1}{2\gamma}\epsilon^{-\frac{1}{2}}(t)\dot{\epsilon}(t)\Big(\frac{a}{2\gamma}{\epsilon}^{-\frac{1}{2}}(t)\|\dot{x}(t)\|^2+\frac{2\gamma}{a}{\epsilon}^{\frac{1}{2}}(t)\|x_{\epsilon(t)}\|^2\Big)$$
and
\begin{eqnarray*}
&&-\big\langle \mathcal{A}_\eta x(t)+\epsilon(t)x(t), \frac{d}{dt}x_{\epsilon(t)}\big\rangle\\
&=&-\big\langle \mathcal{A}_\eta x(t)+\epsilon(t)x_{\epsilon(t)}+\epsilon(t)x(t)-\epsilon(t)x_{\epsilon(t)}, \frac{d}{dt}x_{\epsilon(t)}\big\rangle\\
&=&-\big\langle \mathcal{A}_\eta x(t)+\epsilon(t)x_{\epsilon(t)},\frac{d}{dt}x_{\epsilon(t)}\big\rangle-\epsilon(t)\big\langle x(t)-x_{\epsilon(t)}, \frac{d}{dt}x_{\epsilon(t)}\big\rangle\\
&\leq & \frac{1}{2}\Big(a{\epsilon}^{\frac{1}{2}}(t)\|\mathcal{A}_\eta x(t)+\epsilon(t)x_{\epsilon(t)}\|^2+\frac{1}{a}{\epsilon}^{-\frac{1}{2}}(t)\|\frac{d}{dt}x_{\epsilon(t)}\|^2\Big)\\
&&+ \frac{\epsilon(t)}{2}\Big(c{\epsilon}^{\frac{1}{2}}(t)\| x(t)-x_{\epsilon(t)}\|^2+\frac{1}{c}{\epsilon}^{-\frac{1}{2}}(t)\|\frac{d}{dt}x_{\epsilon(t)}\|^2\Big).
\end{eqnarray*}

Plugging the above inequalities into $(\ref{eq1})$ yields\\
\begin{eqnarray}\label{DE}
 &&\frac{dE(t)}{dt}+\mu(t)E(t)\nonumber\\ 
 &\leq&\epsilon^{\frac{3}{2}}(t)\big[(-\frac{{\gamma}^2}{2}-\frac{2}{{\gamma}^2})\frac{d}{dt}(\frac{1}{\sqrt{\epsilon(t)}})+\delta-2\gamma+\frac{1}{2}{\gamma}^2\delta-\frac{1}{2}{\gamma}^3+\frac{b{\gamma}^2}{2}+\frac{c}{2}\big]\|x(t)-x_{\epsilon(t)}\|^2\nonumber\\ 
 &+&\frac{1}{2}{\epsilon(t)}^{\frac{1}{2}}\big[(1+\frac{a-2}{{\gamma}^2})\frac{d}{dt}(\frac{1}{\sqrt{\epsilon(t)}})+\gamma-\delta\big]\|\dot{x}(t)\|^2\nonumber\\ 
 &+&{\epsilon(t)}^{\frac{1}{2}}\big[(1-\frac{2}{{\gamma}^2})(\rho+\eta)\frac{d}{dt}(\frac{1}{\sqrt{\epsilon(t)}})+((\delta-2\gamma)(\rho+\eta)+\frac{a}{2})\big]\|\mathcal{A}_\eta x(t)+\epsilon(t)x_{\epsilon(t)}\|^2\nonumber\\ 
 &+&\big[(\frac{{\gamma}^2}{2b}+\frac{{\gamma}^3}{4}+\frac{1}{2c}){\epsilon}^{\frac{1}{2}}(t)+\frac{1}{2a}{\epsilon}^{-\frac{1}{2}}(t)\big]\big\|\frac{d}{dt}x_{\epsilon(t)}\big\|^2-\frac{1}{a}\dot{\epsilon}(t)\|x_{\epsilon(t)}\|^2. 
 \end{eqnarray} 
 
 Let us analyze the sign of the coefficients involved in $(\ref{DE})$.
 
  $\bullet$ Since $\lim_{t\rightarrow +\infty}\frac{d}{dt}\big( \frac{1}{\sqrt{\epsilon(t)}}\big)=0$ and $\delta<\gamma+\frac{\gamma}{\frac{1}{2}{\gamma}^2+1}$, we infer 
 $$(-\frac{{\gamma}^2}{2}-\frac{2}{{\gamma}^2})\frac{d}{dt}\big(\frac{1}{\sqrt{\epsilon(t)}}\big)+\delta-2\gamma+\frac{1}{2}{\gamma}^2\delta-\frac{1}{2}{\gamma}^3+\frac{b{\gamma}^2}{2}+\frac{c}{2}\leq 0$$
for all large enough $t$,  by choosing $b$ and $c$ such that  $b{\gamma}^2+c \in (0, -2\delta+4\gamma-{\gamma}^2\delta+{\gamma}^3)$. This gives that the coefficient of $\|x(t)-x_{\epsilon(t)}\|^2$ is nonpositive for all large enough $t$.

  $\bullet$ Since $\lim_{t\rightarrow +\infty}\frac{d}{dt}\big( \frac{1}{\sqrt{\epsilon(t)}}\big)=0$ and $\delta>\gamma$, it holds
$$(1+\frac{a-2}{{\gamma}^2})\frac{d}{dt}(\frac{1}{\sqrt{\epsilon(t)}})+\gamma-\delta \leq 0$$ for all large enough $t$.
  This gives that the coefficient of $\|\dot{x}(t)\|^2$ is nonpositive for all large enough $t$.
   
 $\bullet$ Since $\lim_{t\rightarrow +\infty}\frac{d}{dt}\big( \frac{1}{\sqrt{\epsilon(t)}}\big)=0$ and $\delta<\gamma+\frac{\gamma}{\frac{1}{2}{\gamma}^2+1}<2\gamma$, we deduce 
 $$(1-\frac{2}{{\gamma}^2})(\rho+\eta)\frac{d}{dt}(\frac{1}{\sqrt{\epsilon(t)}})+(\delta-2\gamma)(\rho+\eta)+\frac{a}{2}\leq 0$$
for all large enough $t$, by choosing $a\in (0,2(2\gamma-\delta)(\rho+\eta))$. This gives that the coefficient of $\|\mathcal{A}_\eta x(t)+\epsilon(t)x_{\epsilon(t)}\|^2$ is nonpositive for all large enough $t$.

 Collecting the above estimates, it follows from $(\ref{DE})$ that for all large enough $t$,
 $$\frac{dE(t)}{dt}+\mu(t)E(t)\leq \big[(\frac{{\gamma}^2}{2b}+\frac{{\gamma}^3}{4}+\frac{1}{2c}){\epsilon}^{\frac{1}{2}}(t)+\frac{1}{2a}{\epsilon}^{-\frac{1}{2}}(t)\big]\big\|\frac{d}{dt}x_{\epsilon(t)}\big\|^2-\frac{1}{a}\dot{\epsilon}(t)\|x_{\epsilon(t)}\|^2.$$ 
 Proposition \ref{solu} and Lemma \ref{ineq} guarantee that
 $$\big\|\frac{d}{dt}x_{\epsilon(t)}\big\|^2\leq \frac{{\dot{\epsilon}}^2(t)}{{\epsilon}^2(t)}\|x_{\epsilon(t)}\|^2\leq \frac{{\dot{\epsilon}}^2(t)}{{\epsilon}^2(t)}\|x^*\|^2.$$
 Taking into account the above two relations, we derive
  $$\frac{dE(t)}{dt}+\mu(t)E(t)\leq \big[(\frac{{\gamma}^2}{2b}+\frac{{\gamma}^3}{4}+\frac{1}{2c})\epsilon^{-\frac{3}{2}}(t){\dot{\epsilon}}^2(t)+\frac{1}{2a}\epsilon^{-\frac{5}{2}}(t){\dot{\epsilon}}^2(t)-\frac{1}{a}\dot{\epsilon}(t)\big]\|x^*\|^2$$
for all large enough $t$.
 Now, using the fact that $\epsilon(t)$ is nonincreasing, we conclude that there exists $t_1\geq t_0$ such taht for all $t\geq t_1$  
  $$\frac{dE(t)}{dt}+\mu(t)E(t)\leq \frac{1}{a}\|x^*\|^2\big[\epsilon^{-\frac{5}{2}}(t){\dot{\epsilon}}^2(t)-\dot{\epsilon}(t)\big].$$
 Multiplying this inequality with $\omega(t)=\exp \big(\int_{t_1}^{t}\mu(s)ds\big)$, we deduce
 \begin{eqnarray*}
 \frac{d}{dt}\big(\omega(t)E(t)\big)&=&	\omega(t)\mu(t)E(t)+	\omega(t)\frac{d}{dt}E(t)\\
&\leq& \frac{1}{a}\|x^*\|^2\big[\epsilon^{-\frac{5}{2}}(t){\dot{\epsilon}}^2(t)-\dot{\epsilon}(t)\big]\omega(t).
 \end{eqnarray*}
Hence, by integrating this inequality on $[t_1,t]$, we obtain 
 \begin{eqnarray*}
 E(t)\leq \frac{1}{a}\|x^*\|^2\frac{\int_{t_1}^{t}\big[\epsilon^{-\frac{5}{2}}(s){\dot{\epsilon}}^2(s)-\dot{\epsilon}(s)\big]\omega(s)ds}{\omega(t)}+\frac{\omega(t_1)E(t_1)}{\omega(t)}.
  \end{eqnarray*}
\end{proof} 
 
\begin{theorem}\label{Th4}
Under $(H)$, let $x:[t_0,+\infty)\rightarrow \mathcal{H}$ be a solution trajectory of the system $(\ref{DS})$. Let us assume that  
$$\gamma<\delta <\gamma+\dfrac{\gamma}{\tfrac{{\gamma}^2}{2}+1},\quad \epsilon^{-\frac{1}{2}}(t)\ddot{\epsilon}(t)\leq -\frac{1}{4}(\delta-\gamma)\dot{\epsilon}(t),$$
 and
$$\lim_{t\rightarrow +\infty}{\epsilon}^{-2-\frac{1}{{\gamma}^2}}(t)\dot{\epsilon}(t)=0,\quad \lim_{t\rightarrow +\infty}\exp \big(\int_{t_1}^{t}\sqrt{\epsilon(s)}ds \big)=+\infty.$$
Then, $x(t)$ converges~strongly, as  $t \rightarrow +\infty$, to $x^*$, the element of minimum norm of zer$\mathcal{A}$.

\end{theorem}
\begin{proof} In view of $(\ref{ut})$,
$$\mu(t)=(1-\frac{2}{{\gamma}^2})\frac{d}{dt} \mbox{ln}\frac{1}{\sqrt{\epsilon(t)}}+(\delta-\gamma)\sqrt{\epsilon(t)},$$
from which we deduce that 
 \begin{eqnarray}\label{wt}
 \omega(t)=\exp \bigg( \int_{t_1}^{t} \mu(s)ds \bigg)=\frac{D_1}{{\epsilon}^{\frac{1}{2}-\frac{1}{{\gamma}^2}}(t)}\exp \bigg( \int_{t_1}^{t} (\delta-\gamma)\sqrt{\epsilon(s)}ds \bigg)	
 \end{eqnarray}
for some positive constant $D_1={\epsilon}^{\frac{1}{2}-\frac{1}{{\gamma}^2}}(t_1)$.

The assumptions $\lim_{t\rightarrow +\infty}{\epsilon}^{-2-\frac{1}{{\gamma}^2}}(t)\dot{\epsilon}(t)=0$ and  $\lim_{t\rightarrow +\infty}{\epsilon}(t)=0$ give 
 $$\lim_{t\rightarrow +\infty}\frac{d}{dt}(\frac{1}{\sqrt{\epsilon(t)}})=\lim_{t\rightarrow +\infty}-\frac{1}{2}{\epsilon}^{-\frac{3}{2}}(t)\dot{\epsilon}(t)=0.$$
 Obviously,  all the assumptions of Theorem \ref{Th4.2} are satisfied. As a result,
 \begin{eqnarray*}
 &&E(t)\\
 &\leq&\frac{1}{a}\|x^*\|^2\frac{\int_{t_1}^{t}\big[\epsilon^{-\frac{5}{2}}(s){\dot{\epsilon}}^2(s)-\dot{\epsilon}(s)\big]\omega(s)ds}{\omega(t)}+\frac{\omega(t_1)E(t_1)}{\omega(t)}\\
 &\overset{(\ref{wt})}{=}&\frac{D_1}{a} \|x^*\|^2\frac{\int_{t_1}^{t}\big[\big(\epsilon^{-3+\frac{1}{{\gamma}^2}}(s){\dot{\epsilon}}^2(s)-\epsilon^{-\frac{1}{2}+\frac{1}{{\gamma}^2}}(s)\dot{\epsilon}(s)\big)\exp \big( \int_{t_1}^{s}(\delta-\gamma)\sqrt{\epsilon(\tau)}d\tau\big)\big] ds}{\omega(t)}\\
 &&+\frac{\omega(t_1)E(t_1)}{\omega(t)}	.
 \end{eqnarray*}
Define $\Gamma,\Lambda:[t_1,+\infty)\rightarrow \mathcal{H}$  by
 $$\Gamma(t):=\frac{\int_{t_1}^{t}\big[\epsilon^{-3+\frac{1}{{\gamma}^2}}(s){\dot{\epsilon}}^2(s)\big]\exp \big( \int_{t_1}^{s}(\delta-\gamma)\sqrt{\epsilon(\tau)}d \tau\big)ds}{\omega(t)}$$
 and $$\Lambda(t):=\frac{\int_{t_1}^{t}\big[-\epsilon^{-\frac{1}{2}+\frac{1}{{\gamma}^2}}(s)\dot{\epsilon}(s)\big]\exp \big( \int_{t_1}^{s}(\delta-\gamma)\sqrt{\epsilon(\tau)}d\tau\big)ds}{\omega(t)}.$$
 Then,
 \begin{eqnarray}\label{rE}
 E(t)\leq \frac{D_1}{a}\|x^*\|^2(\Gamma(t)+\Lambda(t))+\frac{\omega(t_1)E(t_1)}{\omega(t)}.	
 \end{eqnarray}

 To estimate $\Gamma(t)$, we observe that 
 \begin{eqnarray*}
 &&\frac{d}{dt}\big [\epsilon^{-\frac{7}{2}+\frac{1}{{\gamma}^2}}(t){\dot{\epsilon}}^2(t)\exp \big(\int_{t_1}^{t}(\delta-\gamma)\sqrt{\epsilon(\tau)}d\tau\big) \big]\\
 &=& \big[(-\frac{7}{2}+\frac{1}{{\gamma}^2}) \epsilon^{-\frac{9}{2}+\frac{1}{{\gamma}^2}}(t){\dot{\epsilon}}^3(t)+2\epsilon^{-\frac{7}{2}+\frac{1}{{\gamma}^2}}(t)\dot{\epsilon}(t)\ddot{\epsilon}(t)+(\delta-\gamma)\epsilon^{-3+\frac{1}{{\gamma}^2}}(t){\dot{\epsilon}}^2(t)\big]\\
 && \cdot \exp \big(\int_{t_1}^{t}(\delta-\gamma)\sqrt{\epsilon(\tau)}d\tau\big)\\
 &=&\epsilon^{-3+\frac{1}{{\gamma}^2}}(t)\dot{\epsilon}(t)\big[(-\frac{7}{2}+\frac{1}{{\gamma}^2})\epsilon^{-\frac{3}{2}}(t){\dot{\epsilon}}^2(t)+(\delta-\gamma)\dot{\epsilon}(t)+2{\epsilon}^{-\frac{1}{2}}(t)\ddot{\epsilon}(t)\big]\\
  && \cdot \exp \big(\int_{t_1}^{t}(\delta-\gamma)\sqrt{\epsilon(\tau)}d\tau\big)\\
 &=&\epsilon^{-3+\frac{1}{{\gamma}^2}}(t)\dot{\epsilon}(t)\big[\dot{\epsilon}(t)\big((7-\frac{2}{{\gamma}^2})\frac{d}{dt}(\frac{1}{\sqrt{\epsilon(t)}})+\delta-\gamma\big)+2{\epsilon}^{-\frac{1}{2}}(t)\ddot{\epsilon}(t)\big]\\
 && \cdot \exp \big(\int_{t_1}^{t}(\delta-\gamma)\sqrt{\epsilon(\tau)}d\tau\big).
 \end{eqnarray*} 
 Since $\lim_{t\rightarrow +\infty}\frac{d}{dt}(\frac{1}{\sqrt{\epsilon(t)}})=0$ and $\dot{\epsilon}(t)\leq 0$, it yields
 $$\dot{\epsilon}(t)\big((7-\frac{2}{{\gamma}^2})\frac{d}{dt}(\frac{1}{\sqrt{\epsilon}(t)})+\delta-\gamma\big)\leq \frac{3}{4}(\delta-\gamma)\dot{\epsilon}(t),$$
 which, in combination with $\epsilon^{-\frac{1}{2}}(t)\ddot{\epsilon}(t)\leq -\frac{1}{4}(\delta-\gamma)\dot{\epsilon}(t)$, yields 
 $$\dot{\epsilon}(t)\big((7-\frac{2}{{\gamma}^2})\frac{d}{dt}(\frac{1}{\sqrt{\epsilon(t)}})+\delta-\gamma\big)+2{\epsilon}^{-\frac{1}{2}}(t)\ddot{\epsilon}(t)\leq \frac{1}{4}(\delta-\gamma)\dot{\epsilon}(t),$$
and from here we infer that 
\begin{eqnarray*}
&&\frac{d}{dt}\big [\epsilon^{-\frac{7}{2}+\frac{1}{{\gamma}^2}}(t){\dot{\epsilon}}^2(t))\exp \big(\int_{t_1}^{t}(\delta-\gamma)\sqrt{\epsilon(\tau)}d\tau\big) \big]\\
&\geq& \frac{1}{4}\epsilon^{-3+\frac{1}{{\gamma}^2}}(t){\dot{\epsilon}}^2(t)\exp \big(\int_{t_1}^{t}(\delta-\gamma)\sqrt{\epsilon(\tau)}d\tau\big).
\end{eqnarray*}
Using $(\ref{wt})$,  we  have 
\begin{eqnarray}\label{Eg}
\Gamma(t)&=&\frac{\int_{t_1}^{t}\big[\epsilon^{-3+\frac{1}{{\gamma}^2}}(s){\dot{\epsilon}}^2(s)\exp \big(\int_{t_1}^{s}(\delta-\gamma)\sqrt{\epsilon(\tau)}d\tau\big) \big]ds}{\omega(t)}\nonumber\\
&\leq &\frac{4\int_{t_1}^{t}\frac{d}{ds}\big[ \epsilon^{-\frac{7}{2}+\frac{1}{{\gamma}^2}}(s){\dot{\epsilon}}^2(s)\exp \big(\int_{t_1}^{s}(\delta-\gamma)\sqrt{\epsilon(\tau)}d\tau\big)\big]ds}{\omega(t)}\nonumber\\
&=&\frac{4\epsilon^{-\frac{7}{2}+\frac{1}{{\gamma}^2}}(t){\dot{\epsilon}}^2(t)\exp \big(\int_{t_1}^{t}(\delta-\gamma)\sqrt{\epsilon(\tau)}d\tau\big)}{\frac{D_1}{{\epsilon(t)}^{\frac{1}{2}-\frac{1}{{\gamma}^2}}}\exp \big( \int_{t_1}^{t} (\delta-\gamma)\sqrt{\epsilon(s)}ds \big)}+\frac{D_2}{\omega(t)}\nonumber\\
&=&\frac{4}{D_1}\epsilon^{-3}(t){\dot{\epsilon}}^2(t)+\frac{D_2}{\omega(t)}
\end{eqnarray}
for some constant $D_2=4\epsilon^{-\frac{7}{2}+\frac{1}{{\gamma}^2}}(t_1){\dot{\epsilon}}^2(t_1)$.

Notice that 
 \begin{eqnarray*}
&& \frac{d}{dt}\big[ -\epsilon^{-1+\frac{1}{{\gamma}^2}}(t)\dot{\epsilon}(t)\exp \big(\int_{t_1}^{t}(\delta-\gamma)\sqrt{\epsilon(\tau)}d\tau\big)  \big]\\
&=&\epsilon^{-\frac{1}{2}+\frac{1}{{\gamma}^2}}(t)\big[(1-\frac{1}{{\gamma}^2})\epsilon^{-\frac{3}{2}}(t){\dot{\epsilon}}^2(t)-{\epsilon}^{-\frac{1}{2}}(t)\ddot{\epsilon}(t)-(\delta-\gamma)\dot{\epsilon}(t)\big]\\
&&\cdot \exp \big(\int_{t_1}^{t}(\delta-\gamma)\sqrt{\epsilon(\tau)}d\tau\big)\\
&=&\epsilon^{-\frac{1}{2}+\frac{1}{{\gamma}^2}}(t)\bigg[-\dot{\epsilon}(t)\bigg((2-\frac{2}{{\gamma}^2})\frac{d}{dt}(\frac{1}{\sqrt{\epsilon(t)}})+\delta-\gamma\bigg)-{\epsilon}^{-\frac{1}{2}}(t)\ddot{\epsilon}(t)\bigg]\\
&&\cdot\exp \big(\int_{t_1}^{t}(\delta-\gamma)\sqrt{\epsilon(\tau)}d\tau\big).
 \end{eqnarray*}
 Using  $\lim_{t\rightarrow+\infty}\frac{d}{dt}(\frac{1}{\sqrt{\epsilon(t)}})=0$ and $\dot{\epsilon}(t)\leq 0$, we get
 $$-\dot{\epsilon}(t)\big((2-\frac{2}{{\gamma}^2})\frac{d}{dt}(\frac{1}{\sqrt{\epsilon(t)}})+\delta-\gamma\big)\geq -\frac{3}{4}(\delta-\gamma)\dot{\epsilon}(t)$$
for all large enough $t$, which, in combination with $\epsilon^{-\frac{1}{2}}(t)\ddot{\epsilon}(t)\leq -\frac{1}{4}(\delta-\gamma)\dot{\epsilon}(t)$, yields  
 $$-\dot{\epsilon}(t)\big[(2-\frac{2}{{\gamma}^2})\frac{d}{dt}(\frac{1}{\sqrt{\epsilon(t)}})+\delta-\gamma\big]-{\epsilon}^{-\frac{1}{2}}(t)\ddot{\epsilon}(t)\geq -\frac{1}{2}(\delta-\gamma)\dot{\epsilon}(t)$$
for all large enough $t$.   It follows that
 \begin{eqnarray*}
 && \frac{d}{dt}\big[ -\epsilon^{-1+\frac{1}{{\gamma}^2}}(t)\dot{\epsilon}(t)\exp \big(\int_{t_1}^{t}(\delta-\gamma)\sqrt{\epsilon(\tau)}d\tau\big)  \big]\\
 &\geq& -\frac{1}{2}(\delta-\gamma)\epsilon^{-\frac{1}{2}+\frac{1}{{\gamma}^2}}(t)\dot{\epsilon}(t)\exp \big(\int_{t_1}^{t}(\delta-\gamma)\sqrt{\epsilon(\tau)}d\tau\big)
 \end{eqnarray*}
for all large enough $t$.
Using $(\ref{wt})$, we have the following estimate for  $\Lambda(t)$: 
\begin{eqnarray}\label{El}
\Lambda(t)&=&\frac{\int_{t_1}^{t}\big[-\epsilon^{-\frac{1}{2}+\frac{1}{{\gamma}^2}}(s)\dot{\epsilon}(s)\exp \big(\int_{t_1}^{s}(\delta-\gamma)\sqrt{\epsilon(\tau)}d\tau\big)\big] ds}{\omega(t)}\nonumber\\
&\leq&\frac{2\int_{t_1}^{t}\frac{d}{ds}\big[ -\epsilon^{-1+\frac{1}{{\gamma}^2}}(s)\dot{\epsilon}(s)\exp (\int_{t_1}^{s}(\delta-\gamma)\sqrt{\epsilon(\tau)}d\tau) \big]ds}{\omega(t)}\nonumber\\
&=&\frac{-2\epsilon^{-1+\frac{1}{{\gamma}^2}}(t)\dot{\epsilon}(t)\exp (\int_{t_1}^{t}(\delta-\gamma)\sqrt{\epsilon(\tau)}d\tau \big)}{\frac{D_1}{{\epsilon(t)}^{\frac{1}{2}-\frac{1}{{\gamma}^2}}}\exp \big( \int_{t_1}^{t} (\delta-\gamma)\sqrt{\epsilon(s)}ds \big)}+\frac{D_3}{\omega(t)}\nonumber\\
&=&-\frac{2}{D_1}\epsilon^{-\frac{1}{2}}(t)\dot{\epsilon}(t)+\frac{D_3}{\omega(t)}
\end{eqnarray}
for some constant $D_3=-2\epsilon^{-1+\frac{1}{{\gamma}^2}}(t_1)\dot{\epsilon}(t_1)$.

Plugging $(\ref{Eg})$ and $(\ref{El})$ into $(\ref{rE})$ yields
$$E(t)\leq  \frac{4\|x^*\|^2}{a}\epsilon^{-3}(t){\dot{\epsilon}}^2(t)-\frac{2\|x^*\|^2}{a}\epsilon^{-\frac{1}{2}}(t)\dot{\epsilon}(t)+\frac{\bar{D}}{\omega(t)}$$ 
forall large enolugh $t$, where $\bar{D}=\omega(t_1)E(t_1)+\frac{D_1(D_2+D_3)}{a}\|x^*\|^2$ is a positive constant. By  Theorem \ref{Th4.1}, we deduce that 
$$\|x(t)-x_{\epsilon(t)}\|^2\leq \frac{E(t)}{\epsilon(t)}\leq\frac{4\|x^*\|^2}{a}\epsilon^{-4}(t){\dot{\epsilon}}^2(t)-\frac{2\|x^*\|^2}{a}\epsilon^{-\frac{3}{2}}(t)\dot{\epsilon}(t)+\frac{\bar{D}}{\epsilon(t)\omega(t)}.$$
Notice the assmptions $\lim_{t\rightarrow +\infty}\epsilon(t)=0$ and $\lim_{t\rightarrow +\infty}\epsilon^{-2-\frac{1}{{\gamma}^2}}(t)\dot{\epsilon}(t)=0$ yield 
$$\lim_{t\rightarrow +\infty}\epsilon^{-\frac{3}{2}}(t)\dot{\epsilon}(t)=0.$$
In order to prove $\lim_{t\rightarrow +\infty}\|x(t)-x_{\epsilon(t)}\|^2=0$, we just need to show $\lim_{t\rightarrow +\infty}\epsilon(t)\omega(t)=+\infty$.
Since $\lim_{t\rightarrow +\infty}\exp \big( \int_{t_1}^{t}\sqrt{\epsilon(s)}ds\big)=+\infty$ and $\lim_{t\rightarrow +\infty}\epsilon(t)=0$, we arrive at
\begin{eqnarray*}
&&\lim_{t\rightarrow +\infty}\epsilon(t)\omega(t)\\
&\overset{(\ref{wt})}{=}&\lim_{t\rightarrow +\infty}{\epsilon(t)}^{\frac{1}{2}+\frac{1}{{\gamma}^2}}\exp \big((\delta-\gamma) \int_{t_1}^{t} \sqrt{\epsilon(s)}ds\big)\\
&=&\lim_{t\rightarrow +\infty} \frac{\exp \big((\delta-\gamma) \int_{t_1}^{t} \sqrt{\epsilon(s)}ds\big)}{{\epsilon(t)}^{-\frac{1}{2}-\frac{1}{{\gamma}^2}}}\\
&\overset{L'Hopital's}{=}&\lim_{t\rightarrow +\infty} \frac{(\delta-\gamma)\exp \big((\delta-\gamma) \int_{t_1}^{t}\sqrt{\epsilon(s)}ds\big)}{(-\frac{1}{2}-\frac{1}{{\gamma}^2}){\epsilon}^{-2-\frac{1}{{\gamma}^2}}(t)\dot{\epsilon}(t)}\\
&=&+\infty,
\end{eqnarray*} 
where the last equality is from the assumption $\lim_{t\rightarrow +\infty}{\epsilon}^{-2-\frac{1}{{\gamma}^2}}(t)\dot{\epsilon}(t)=0$.
According to Proposition $\ref{solu}$, $x(t)$ converges strongly to $x^*$ as $t\rightarrow +\infty$.
\end{proof}
  
\section{Particular cases}\label{Pc}
When $\epsilon(t)=\frac{1}{t^q}$ with $0<q<1$, the system \eqref{DS} reduces to the following system
\begin{equation}\label{TRDS}
\ddot{x}(t)+\frac{\delta}{t^{\frac{q}{2}}}\dot{x}(t)+\frac{1}{\gamma}t^{\frac{q}{2}}\frac{d}{dt}\big(\mathcal{A}_{\eta}x(t)+\frac{1}{t^q}x(t)\big)+\mathcal{A}_{\eta}x(t)+\frac{1}{t^q}x(t)=0.
\end{equation}
In this section, we further establish   some convergence rate results  for the system \eqref{TRDS}.

\begin{theorem}
Under (H), let $x:[t_0, +\infty)\rightarrow \mathcal{H}$ be a solution trajectory of the system \eqref{TRDS}.
Let us assume that   $0<q<1$ and  $\gamma<\delta <\gamma+\frac{\gamma}{\frac{1}{2}{\gamma}^2+1}$. Then, the following conclusions are true:
\begin{itemize}
\item[i)] For $0<q<\frac{2}{3}$, it holds that
$$\|x(t)-x_{\epsilon(t)}\|^2=O\big(t^{\frac{q}{2}-1}\big),\quad\|\dot{x}(t)\|^2=O\big( t^{q-2}\big),$$
$$\|\mathcal{A}_\eta x(t)+\epsilon(t)x_{\epsilon(t)}\|^2=O\big( t^{q-2}\big)~\mbox{as}~t\rightarrow +\infty.$$
\item[ii)] For $\frac{2}{3} \leq q<1$, it holds that
$$\|x(t)-x_{\epsilon(t)}\|^2=O\big(t^{2q-2}\big),\quad\|\dot{x}(t)\|^2=O\big( t^{-\frac{q}{2}-1}\big),$$
$$\|\mathcal{A}_\eta x(t)+\epsilon(t)x_{\epsilon(t)}\|^2=O\big( t^{-\frac{q}{2}-1}\big)~\mbox{as}~t\rightarrow +\infty.$$
\end{itemize} 
\end{theorem}
 
\begin{proof}
 By $(\ref{ut})$ and $(\ref{wt})$, we have
$$\mu(t)=(\frac{1}{2}-\frac{1}{{\gamma}^2})qt^{-1}+(\delta-\gamma)t^{-\frac{q}{2}},$$
and 
\begin{eqnarray}\label{wtt}
	\omega(t)=\exp\big(\int_{t_1}^{t}\mu(s)ds\big)&=&\exp\big[\int_{t_1}^{t}(\frac{1}{2}-\frac{1}{{\gamma}^2})qs^{-1}+(\delta-\gamma)s^{-\frac{q}{2}}ds\big]\nonumber\\
	&=&\big(\frac{t}{t_1}\big)^{(\frac{1}{2}-\frac{1}{{\gamma}^2})q}\exp\big[\frac{2(\delta-\gamma)}{2-q}(t^{\frac{2-q}{2}}-t^{\frac{2-q}{2}}_{1})\big]\nonumber\\
	&=&N_1 t^{(\frac{1}{2}-\frac{1}{{\gamma}^2})q}\exp\big(\delta_0 t^{\frac{2-q}{2}}\big),
\end{eqnarray}
where $N_1=\big[t^{(\frac{1}{2}-\frac{1}{{\gamma}^2})q}_{1}\exp(\delta_0 t^{\frac{2-q}{2}}_{1})\big]^{-1}$ and $\delta_0=\frac{2(\delta-\gamma)}{2-q}$.
We can verify that in the case $\epsilon(t)=\frac{1}{t^q}$ with $0<q<1$,
$$\lim_{t\rightarrow +\infty}\frac{d}{dt}(\frac{1}{\sqrt{\epsilon(t)}})=-\frac{1}{2}\lim_{t\rightarrow +\infty}{\epsilon}^{-\frac{3}{2}}(t)\dot{\epsilon}(t)=\frac{q}{2}\lim_{t\rightarrow +\infty} t^{\frac{q}{2}-1}=0.$$
Therefore, all the  assumptions of Theorem \ref{Th4.2} are satisfied. According to Theorem \ref{Th4.2}, we get
 $$E(t)\leq \frac{1}{a}E_1(t)\|x^*\|^2+E_2(t),$$
 where
\begin{eqnarray}\label{EE1}
E_1(t)=\frac{1}{\omega(t)}\int_{t_1}^{t}\bigg (\epsilon^{-\frac{5}{2}}(s){\dot{\epsilon}}^2(s)-\dot{\epsilon}(s)\bigg ) \omega(s)ds
\end{eqnarray}

and 
$$E_2(t)=\frac{\omega(t_1)E(t_1)}{\omega(t)}=\frac{\omega(t_1)E(t_1)}{N_1} t^{(-\frac{1}{2}+\frac{1}{{\gamma}^2})q}\exp\big(-\delta_0 t^{\frac{2-q}{2}}\big).$$
Since $0<q<1$ and $\delta>\gamma$, we infer that $E_2(t)$ has an exponential decay to zero, as $t\rightarrow +\infty$. As a consequence, there exists a positive constant $M$ such that
\begin{eqnarray}\label{ME1}
E(t)\leq M\|x^*\|^2 E_1(t)	
\end{eqnarray}
for all large enough $t$.

We only have to focus on the asymptotic behavior of $E_1(t)$.

$i)$: In the case $0<q<\frac{2}{3}$, we get $t^{q-2}< t^{-\frac{q}{2}-1}$ for all large enough $t$. It follows that  
\begin{equation}\label{f26-1}
E_1(t)\leq 2qt^{(\frac{1}{{\gamma}^2}-\frac{1}{2})q}\exp \big(-\delta_0 t^{\frac{2-q}{2}}\big)\int_{t_1}^{t}\big(s^{-\frac{q}{2}-\frac{q}{{\gamma}^2}-1} \big)\exp\big( \delta_0 s^{\frac{2-q}{2}}\big) ds.
\end{equation}
Let us estimate the integral $\int_{t_1}^{t}s^{-\frac{q}{2}-\frac{q}{{\gamma}^2}-1}\exp\big(\delta_0 s^{\frac{2-q}{2}}\big)ds$.
Observe further that
\begin{eqnarray}\label{f26-2}
&&\frac{d}{dt}\bigg (t^{-\frac{q}{{\gamma}^2}-1}\exp \big(\delta_0 t^{\frac{2-q}{2}}\big)\bigg )\nonumber\\
&=& \big[(-\frac{q}{{\gamma}^2}-1)t^{-\frac{q}{{\gamma}^2}-2}+\frac{2-q}{2}\delta_0t^{-\frac{q}{{\gamma}^2}-1-\frac{q}{2}}\big]\exp \big(\delta_0 t^{\frac{2-q}{2}}\big)\nonumber\\
&=& \big[(-\frac{q}{{\gamma}^2}-1)t^{-\frac{q}{{\gamma}^2}-2}+\frac{2-q}{4}\delta_0t^{-\frac{q}{{\gamma}^2}-1-\frac{q}{2}}+\frac{2-q}{4}\delta_0t^{-\frac{q}{{\gamma}^2}-1-\frac{q}{2}}\big]\cdot\exp \big(\delta_0 t^{\frac{2-q}{2}}\big)\nonumber\\
&\geq& \frac{2-q}{4}\delta_0t^{-\frac{q}{2}-\frac{q}{{\gamma}^2}-1}\exp \big(\delta_0 t^{\frac{2-q}{2}}\big),
\end{eqnarray}
where the inequality is due to $t^{-\frac{q}{{\gamma}^2}-2}<t^{-\frac{q}{{\gamma}^2}-1-\frac{q}{2}}$ for $0<q<1$. Combining \eqref{f26-1} and \eqref{f26-2}, we infer that 
 \begin{eqnarray*}
 E_1(t)&\leq&\frac{8q}{(2-q)\delta_0}t^{(\frac{1}{{\gamma}^2}-\frac{1}{2})q}\exp \big(-\delta_0 t^{\frac{2-q}{2}}\big)\int_{t_1}^{t} \frac{d}{ds}\bigg (s^{-\frac{q}{{\gamma}^2}-1}\exp \big(\delta_0 s^{\frac{2-q}{2}}\big)\bigg )ds\\
&=& \frac{8q}{(2-q)\delta_0t^{(\frac{1}{2}-\frac{1}{{\gamma}^2})q}\exp \big(\delta_0 t^ {\frac{2-q}{2}}\big)}t^{-\frac{q}{{\gamma}^2}-1}\exp \big(\delta_0 t^ {\frac{2-q}{2}}\big)\\
&&-\frac{8q}{(2-q)\delta_0t^{(\frac{1}{2}-\frac{1}{{\gamma}^2})q}\exp \big(\delta_0 t^ {\frac{2-q}{2}}\big)}t_1^{-\frac{q}{{\gamma}^2}-1}\exp \big(\delta_0 t_1^ {\frac{2-q}{2}}\big)\\
&\leq & \frac{8q}{(2-q)\delta_0}t^{-\frac{q}{2}-1}.
 \end{eqnarray*}
This together with $(\ref{ME1})$ yields
$$E(t)\leq  \frac{8qM}{(2-q)\delta_0}\|x^*\|^2t^{-\frac{q}{2}-1}$$
for all large enough $t$. Dividing it by $\epsilon(t)=\frac{1}{t^q}$, we have
$$\frac{E(t)}{\epsilon(t)}\leq \frac{8qM}{(2-q)\delta_0}\|x^*\|^2t^{\frac{q}{2}-1}.$$
In view of Theorem $\ref{Th4.1}$, we obtain at once that  
$$\|x(t)-x_{\epsilon(t)}\|^2=O(t^{\frac{q}{2}-1}),\qquad\|\dot{x}(t)\|^2=O(t^{-\frac{q}{2}-1}),$$
and 
$$\|\mathcal{A}_\eta x(t)+\epsilon(t)x_{\epsilon(t)}\|^2=O\big( t^{-\frac{q}{2}-1}\big)$$
as $t\rightarrow +\infty$.
 
$ii)$: In the case $\frac{2}{3}\leq q<1$, we get $t^{q-2}\geq t^{-\frac{q}{2}-1}$ for all large enough $t$. It follows from $(\ref{EE1})$  that
\begin{eqnarray}\label{f26-3}
E_1(t)&=&t^{(\frac{1}{{\gamma}^2}-\frac{1}{2})q}\exp \big(-\delta_0 t^{\frac{2-q}{2}}\big)\int_{t_1}^{t}\big(q^2 s^{q-\frac{q}{{\gamma}^2}-2}+q s^{-\frac{q}{2}-\frac{q}{{\gamma}^2}-1} \big)\exp\big( \delta_0 s^{\frac{2-q}{2}}\big) ds\nonumber\\
&\leq&	2q^2t^{(\frac{1}{{\gamma}^2}-\frac{1}{2})q}\exp \big(-\delta_0 t^{\frac{2-q}{2}}\big)\int_{t_1}^{t}\big(s^{q-\frac{q}{{\gamma}^2}-2} \big)\exp\big( \delta_0 s^{\frac{2-q}{2}}\big) ds.
\end{eqnarray}
Let us estimate the integral $\int_{t_1}^{t} s^{(1-\frac{1}{{\gamma}^2})q-2}\exp \big(\delta_0 s^{\frac{2-q}{2}}\big)ds$. Notice that 
\begin{eqnarray}\label{f26-4}
&&\frac{d}{dt}\bigg (t^{\frac{3q}{2}-\frac{q}{{\gamma}^2}-2}\exp \big(\delta_0 t^{\frac{2-q}{2}}\big)\bigg )\nonumber\\
&=&\big[ \big (\frac{3q}{2}-\frac{q}{{\gamma}^2}-2\big ) t^{\frac{3q}{2}-\frac{q}{{\gamma}^2}-3}+\frac{2-q}{2}\delta_0t^{(1-\frac{1}{{\gamma}^2})q-2}\big]\exp \big(\delta_0 t^{\frac{2-q}{2}}\big)\nonumber\\
&=& \big[\big (\frac{3q}{2}-\frac{q}{{\gamma}^2}-2\big ) t^{\frac{3q}{2}-\frac{q}{{\gamma}^2}-3}+\frac{2-q}{4}\delta_0t^{(1-\frac{1}{{\gamma}^2})q-2}+\frac{2-q}{4}\delta_0t^{(1-\frac{1}{{\gamma}^2})q-2}\big]\exp \big(\delta_0 t^{\frac{2-q}{2}}\big)\nonumber\\
&\geq& \frac{2-q}{4}\delta_0t^{(1-\frac{1}{{\gamma}^2})q-2}\exp \big(\delta_0 t^{\frac{2-q}{2}}\big),
\end{eqnarray}
where the inequality is from the fact that $t^{\frac{3q}{2}-\frac{q}{{\gamma}^2}-3}<t^{(1-\frac{1}{{\gamma}^2})q-2}$ for $0<q<1$. It follows from \eqref{f26-3} and \eqref{f26-4} that for $t$ large enough, 
\begin{eqnarray*}
E_1(t)&\leq&\frac{8q^2}{(2-q)\delta_0 t^{(\frac{1}{2}-\frac{1}{{\gamma}^2})q}\exp \big(\delta_0 t^ {\frac{2-q}{2}}\big)}\int_{t_1}^{t} \frac{d}{ds}\bigg (s^{(1-\frac{1}{{\gamma}^2})q+\frac{q}{2}-2}\exp \big(\delta_0 s^{\frac{2-q}{2}}\big)\bigg )ds\\
&=& \frac{8q^2}{(2-q)\delta_0 t^{(\frac{1}{2}-\frac{1}{{\gamma}^2})q}\exp \big(\delta_0 t^ {\frac{2-q}{2}}\big)}t^{(1-\frac{1}{{\gamma}^2})q-2+\frac{q}{2}}\exp \big(\delta_0 t^ {\frac{2-q}{2}}\big)\\
&&-\frac{8q^2}{(2-q)\delta_0 t^{(\frac{1}{2}-\frac{1}{{\gamma}^2})q}\exp \big(\delta_0 t^ {\frac{2-q}{2}}\big)}t_1^{(1-\frac{1}{{\gamma}^2})q-2+\frac{q}{2}}\exp \big(\delta_0 t_1^ {\frac{2-q}{2}}\big)\\
&\leq & \frac{8q^2}{(2-q)\delta_0}t^{q-2},
\end{eqnarray*} 
which, in combination with $(\ref{ME1})$, further yields
$$E(t)\leq  \frac{8q^2M}{(2-q)\delta_0}\|x^*\|^2t^{q-2}.$$
After dividing it by $\epsilon(t)=\frac{1}{t^q}$, we deduce that
$$\frac{E(t)}{\epsilon(t)}\leq \frac{8q^2M}{\delta_0(2-q)}\|x^*\|^2 t^{2q-2}$$
for all large enough $t$. According to Theorem $\ref{Th4.1}$, we obtain at once that  
$$\|x(t)-x_{\epsilon(t)}\|^2=O(t^{2q-2}),\qquad\|\dot{x}(t)\|^2=O(t^{q-2}),$$
and 
$$\|\mathcal{A}_\eta x(t)+\epsilon(t)x_{\epsilon(t)}\|^2=O\big( t^{q-2}\big)$$
as $ t\rightarrow +\infty$.
 
\end{proof}

 \section{Numerical illustration}\label{Ex}
In this section, we illustrate the validity of the  system $(\ref{DS})$ by an example. The simulations are conducted in Matlab (version 9.4.0.813654)R2018a. All the numerical procedures are performed on a personal computer with Inter(R) Core(TM) i7-4600U, CORES 2.69GHz and RAM 8.00GB.
All the dynamical systems in this section are solved numerically by the ode45 function in Matlab on the interval $[0.1,100]$.

\begin{example}\label{Ex1} Consider the following inclusion problem
$$0\in \mathcal{A}(x),$$
 where $\mathcal{A}=
\begin{pmatrix}
1 & 0 & 0\\
0 & 0 & 0\\
0 & 0 & -1
\end{pmatrix}$. It is easily verified that the set of solutions is $\{(0,b,0)^T:b\in \mathbb{R}\}$ and $x^*=(0,0,0)^T$ is the minimum norm solution. We will show that $\mathcal{A}$ is a maximally $\rho$-comonotone operator with $\rho=-1$. According to  $(i)$ of Proposition \ref{z}, we just need to prove ${\mathcal{A}}^{-1}+Id$ is a maximally monotone operator. Let $ x=(x_1,x_2,x_3)^T$ and $u=(u_1,u_2,u_3)^T$  such  that $u=\mathcal{A} x=(x_1,0,-x_3)^T$. Since $x\in {\mathcal{A}}^{-1}u\Leftrightarrow u=\mathcal{A}x$,  it follows that  $\text{Dom} {\mathcal{A}}^{-1}=\{(u_1,0,u_3)^T: u_1,u_3 \in \mathbb{R}\}$ and
\begin{equation}\label{FFYP}{\mathcal{A}}^{-1}u=\begin{cases}
\emptyset, \qquad\qquad\qquad\qquad\qquad~u\notin \text{Dom} {\mathcal{A}}^{-1};\\
\{(u_1,\lambda,-u_3)^T:\lambda\in \mathbb{R}\}, ~~u\in\text{Dom} {\mathcal{A}}^{-1}.
\end{cases}
\end{equation}
Let $ u=(u_1,0,u_3)^T,~v=(v_1,0,v_3)^T\in Dom {\mathcal{A}}^{-1}$,  and let $y\in ({\mathcal{A}}^{-1}+Id)u$, $z\in ({\mathcal{A}}^{-1}+Id)v$. Then there exist $\lambda_1,\lambda_2 \in \mathbb{R}$ such that $y=(2u_1,\lambda_1,0)^T$ and $z=(2v_1,\lambda_2,0)^T$. We can infer that
$$\langle y-z, u-v \rangle =2\|u_1-v_1\|^2\geq 0.$$
Therefore, ${\mathcal{A}}^{-1}+Id$ is a monotone operator.  To justify the maximality of  ${\mathcal{A}}^{-1}+Id$, let $\bar u=(\bar u_1,0,\bar u_3)^T \in \text{Dom}({\mathcal{A}}^{-1}+Id)$ and $\bar x=(\bar x_1,\bar x_2,\bar x_3)^T$ such that
$$\langle y-\bar x, u-\bar u\rangle \geq 0,\qquad \forall (u,y)\in \text{gra}({\mathcal{A}}^{-1}+Id).$$ 
This together with \eqref{FFYP} implies that
\begin{eqnarray*}
\langle y-\bar x, u-\bar u\rangle =(2u_1-\bar{x}_1)(u_1-\bar{u}_1)	-\bar{x}_3(u_3-\bar{u}_3)\geq 0,\qquad\forall u_1,u_3\in \mathbb{R}.
\end{eqnarray*}
This yields  $\bar{x}_3=0$ and $\bar{x}_1=2\bar{u}_1$. Again using  \eqref{FFYP}, we have $\bar x\in ({\mathcal{A}}^{-1}+Id)\bar{u}$.
Therefore, ${\mathcal{A}}^{-1}+Id$ is a maximally monotone, and so $\mathcal{A}$ is a maximally $\rho$-comonotone operator with $\rho=-1$.

Next, we test the behaviors of the  system $(\ref{DS})$ on Example \ref{Ex1}. 

We take $\gamma=1$, $\delta=\frac{4}{3}$, $\eta=3$ and $\epsilon(t)=t^{-q}$ in the system $(\ref{DS})$ with $x(t_0)=(1,1,1)^T$ and $\dot{x}(t_0)=(1,2,3)^T$. Figure \ref{Figure1} depicts the asymptotical behavior of the trajectory $x(t)$, the velocity $\dot{x}(t)$ and the Yosida regularization $\mathcal{A}_\eta x(t)$ generated by $(\ref{DS})$ with parameter $q=\frac{1}{5}$, $q=\frac{1}{3}$ and $q=\frac{1}{2}$, respectively. Figure \ref{Figure2} depicts the asymptotical behavior of the trajectory $x(t)$, the velocity $\dot{x}(t)$ and the Yosida regularization $\mathcal{A}_\eta x(t)$ generated by $(\ref{DS})$ with parameter $q=\frac{2}{3}$, $q=\frac{3}{4}$ and $q=\frac{5}{6}$, respectively.

Next, we compare the behaviorof the trajectory  $x(t)$  of the system  $\eqref{DS}$ with that of the systen $\eqref{Tds}$ on  Example \ref{Ex1}. We take $\gamma=1$, $\delta=\frac{4}{3}$, $\eta=3$ and $\epsilon(t)=t^{-\frac{1}{2}}$ in the system $(\ref{DS})$ and take $\alpha=\frac{4}{3}$, $\beta=1$ and $\eta=3$ in the system $(\ref{Tds})$ with $x(t_0)=(1,1,1)^T$ and $\dot{x}(t_0)=(1,2,3)^T$. As shown in Figure \ref{Figure3}, the trajectory $x(t)$ of the system $(\ref{DS})$ converges to the  minimum norm solution $x^*$, while the trajectory $x(t)$ of the system  $\eqref{Tds}$ converges a solution which need not to be the  minimum norm solution.


\end{example}

\begin{figure*}[h]
 \centering
 {
  \begin{minipage}[t]{0.31\textwidth}
   \centering
   \includegraphics[width=1.8in]{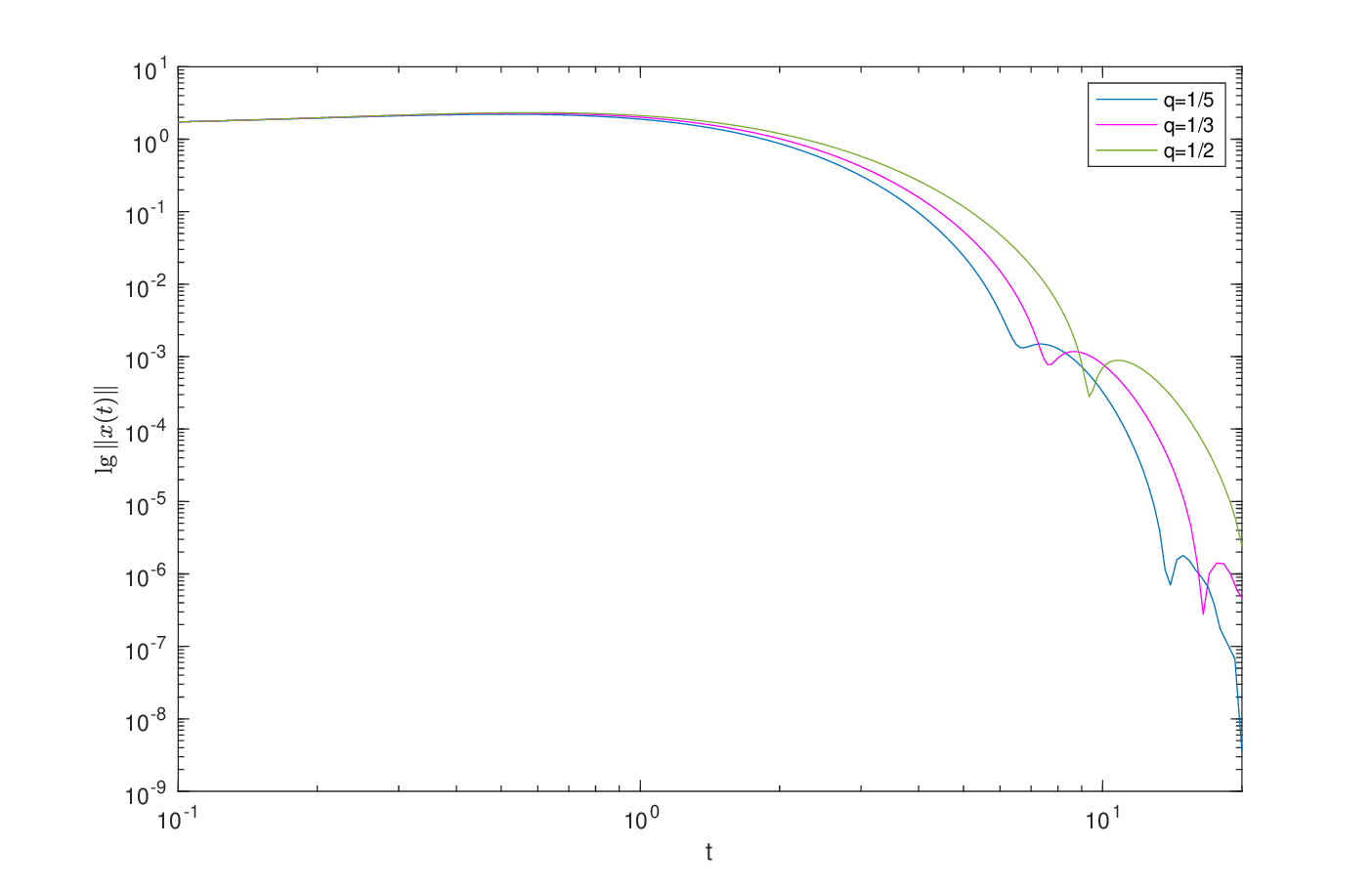}
  \end{minipage}
 }
 {
  \begin{minipage}[t]{0.31\textwidth}
   \centering
   \includegraphics[width=1.8in]{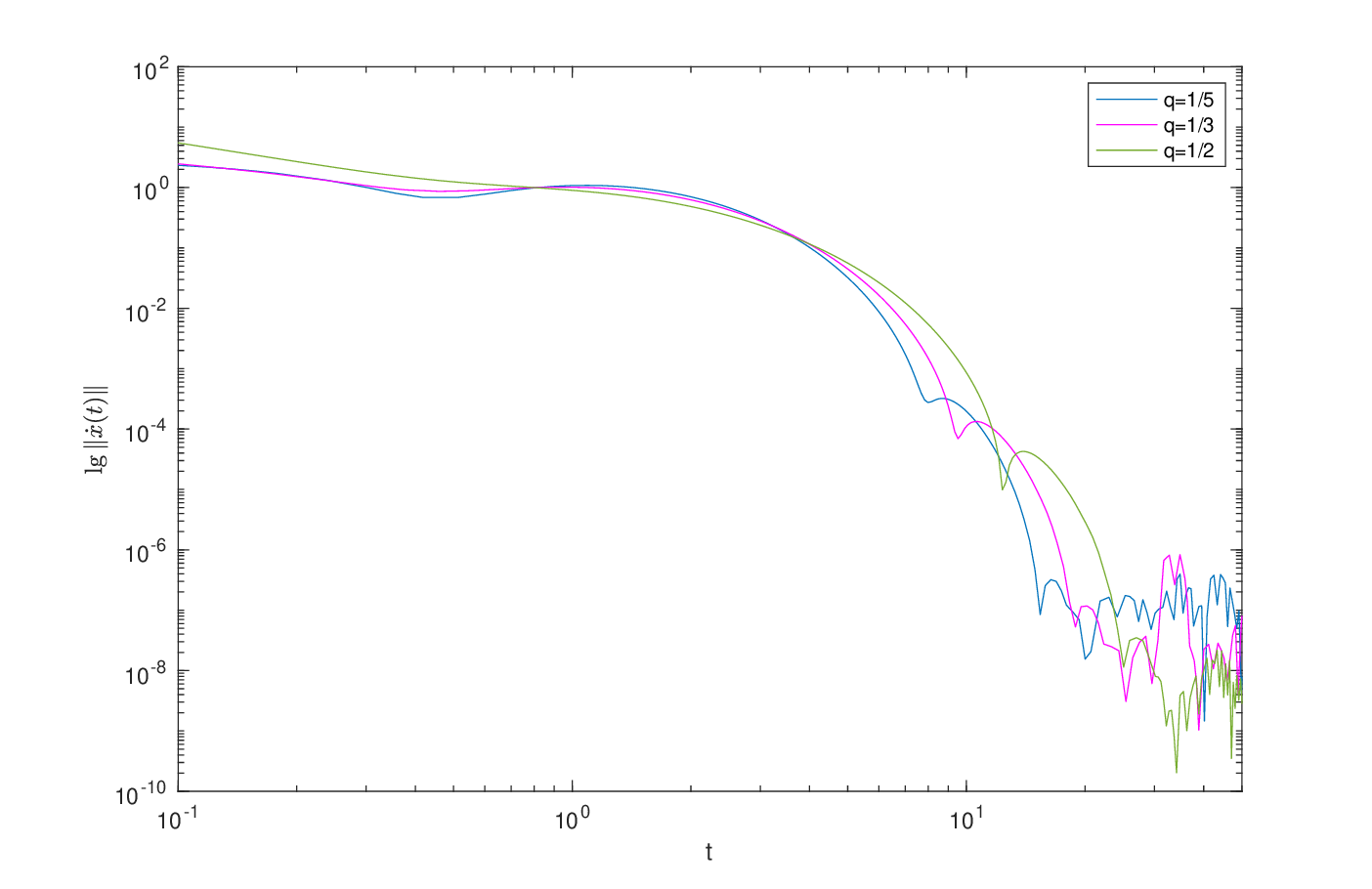}
  \end{minipage}
 }
 {
  \begin{minipage}[t]{0.31\textwidth}
   \centering
   \includegraphics[width=1.8in]{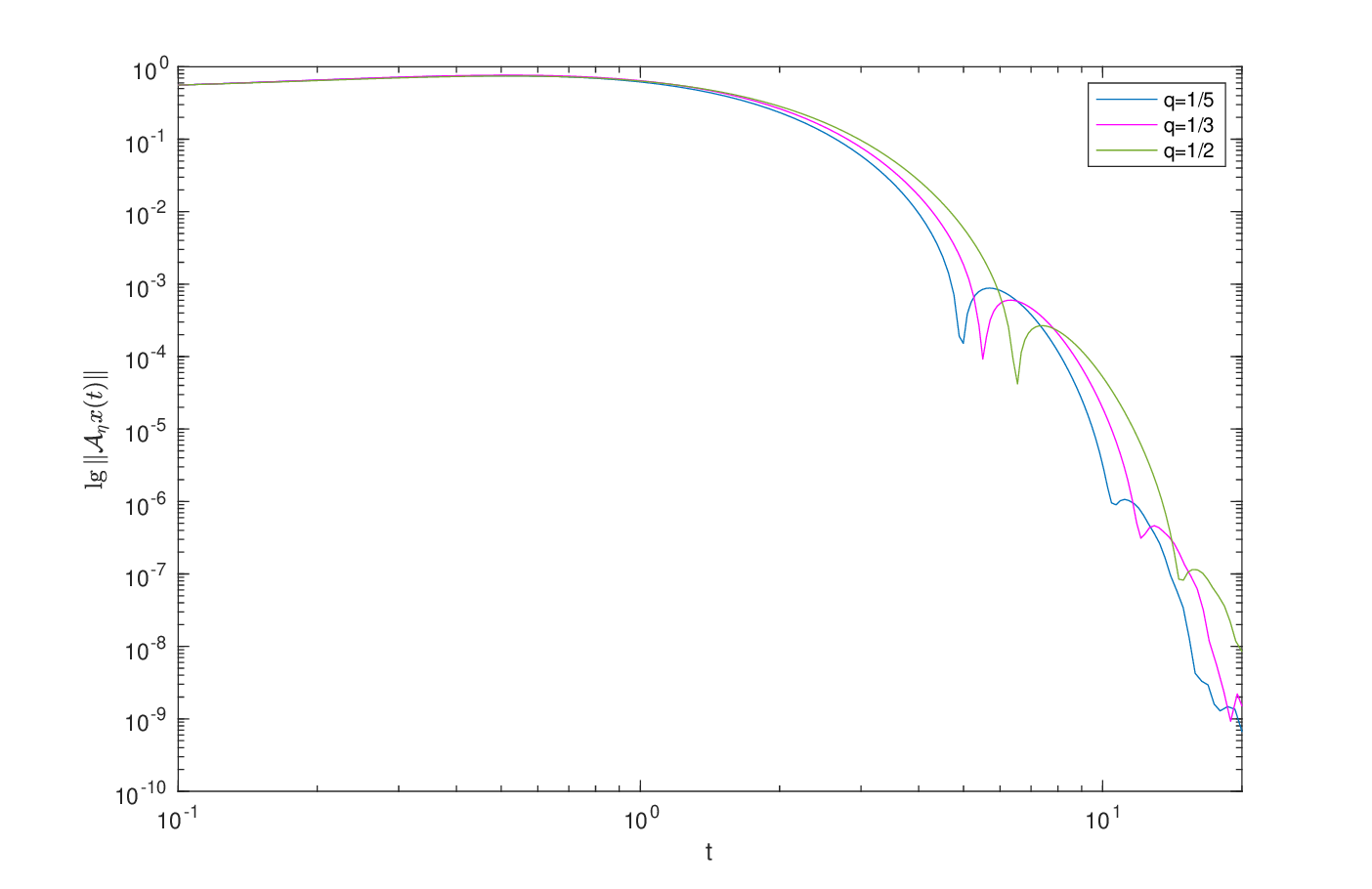}
  \end{minipage}
 }
 \caption{Rescaled iteration errors $\text{ lg } \|x(t)-x^*\|$, $\text{ lg } \|\dot{x}(t)\|$, $\text{ lg }\|\mathcal{A}_{\eta}x(t)\|$ of \eqref{DS} in Example \ref{Ex1}.}
 \label{Figure1}
 \centering
\end{figure*}

\begin{figure*}[h]
 \centering
 {
  \begin{minipage}[t]{0.31\textwidth}
   \centering
   \includegraphics[width=1.8in]{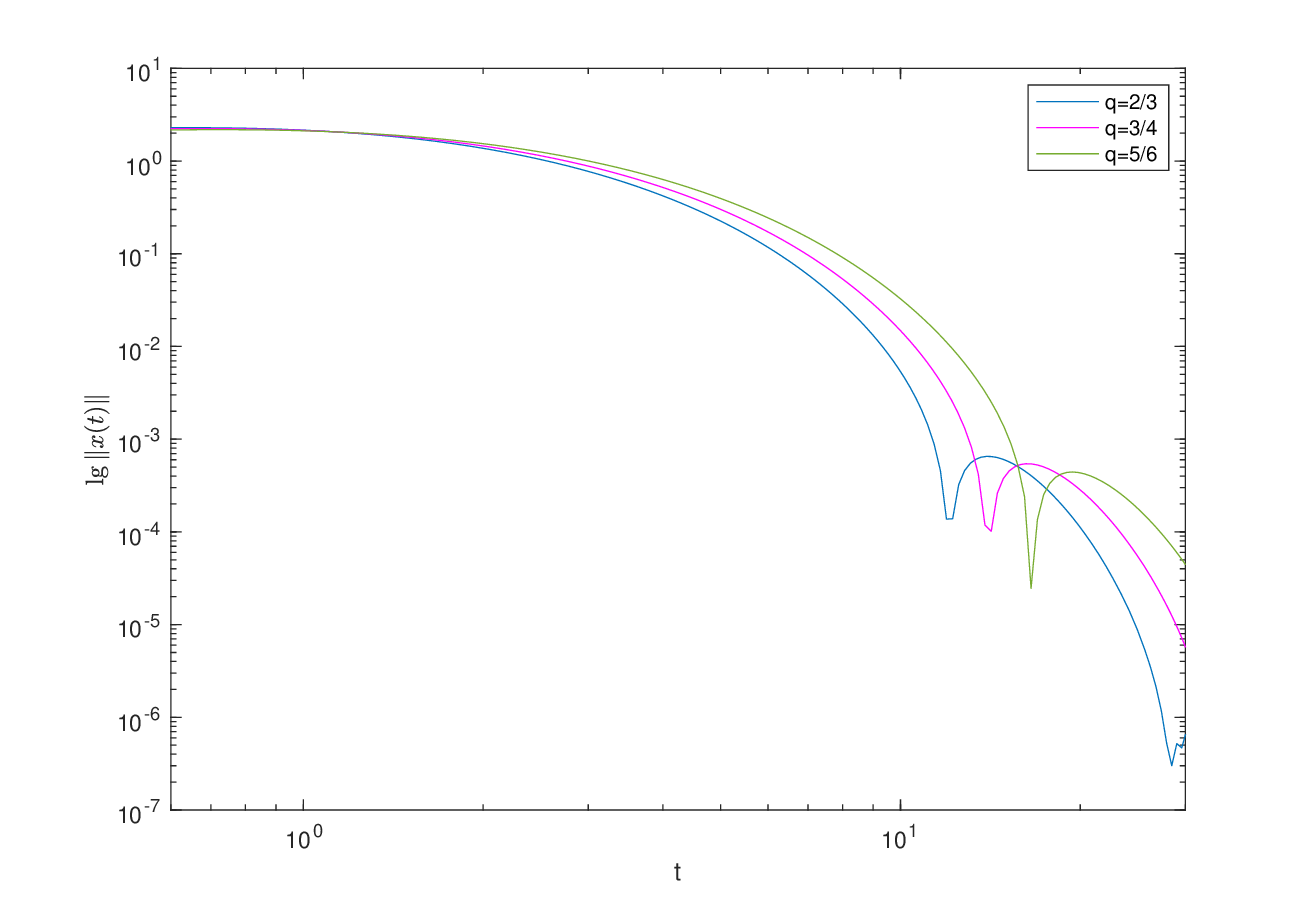}
  \end{minipage}
 }
 {
  \begin{minipage}[t]{0.31\textwidth}
   \centering
   \includegraphics[width=1.8in]{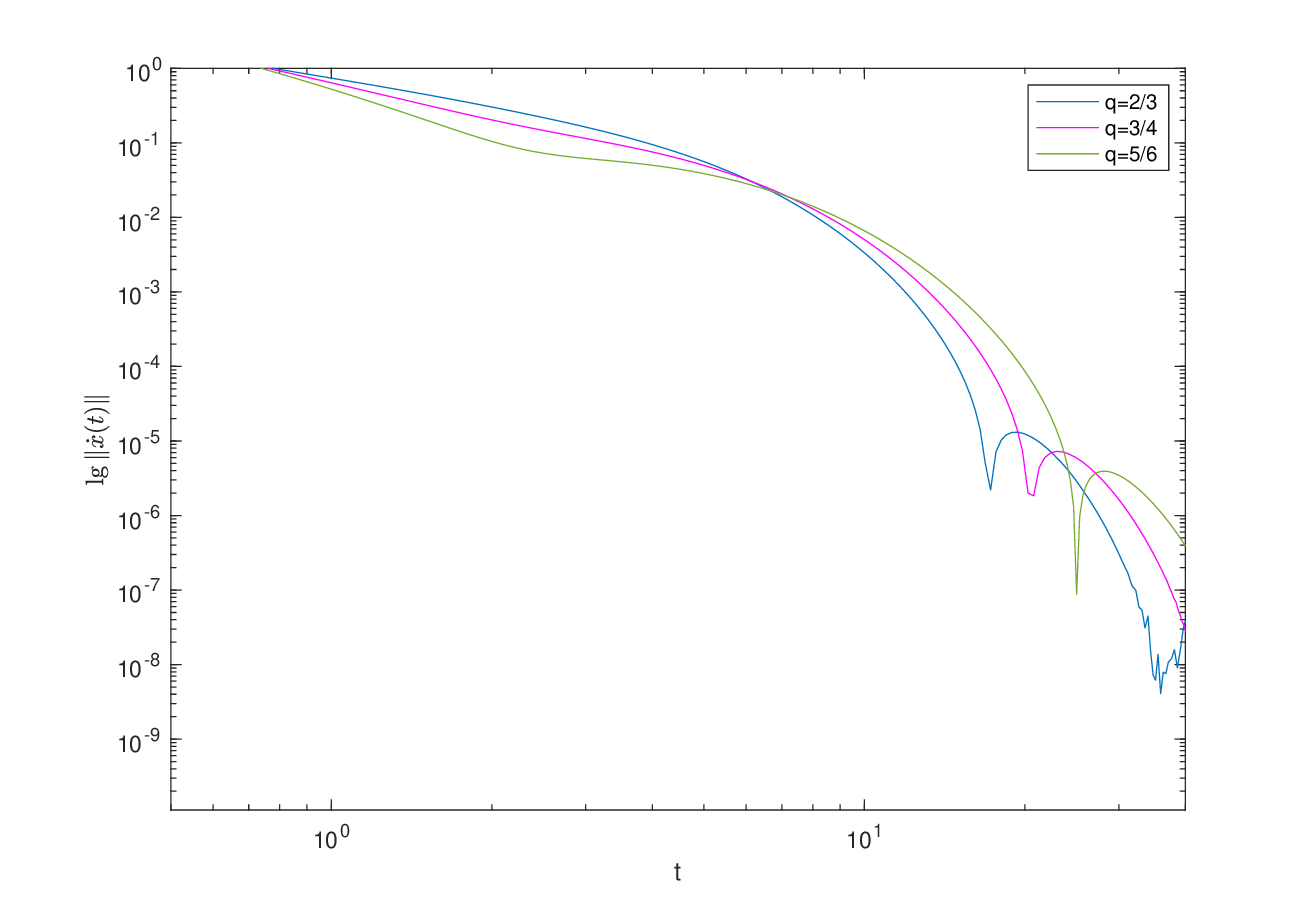}
  \end{minipage}
 }
 {
  \begin{minipage}[t]{0.31\textwidth}
   \centering
   \includegraphics[width=1.8in]{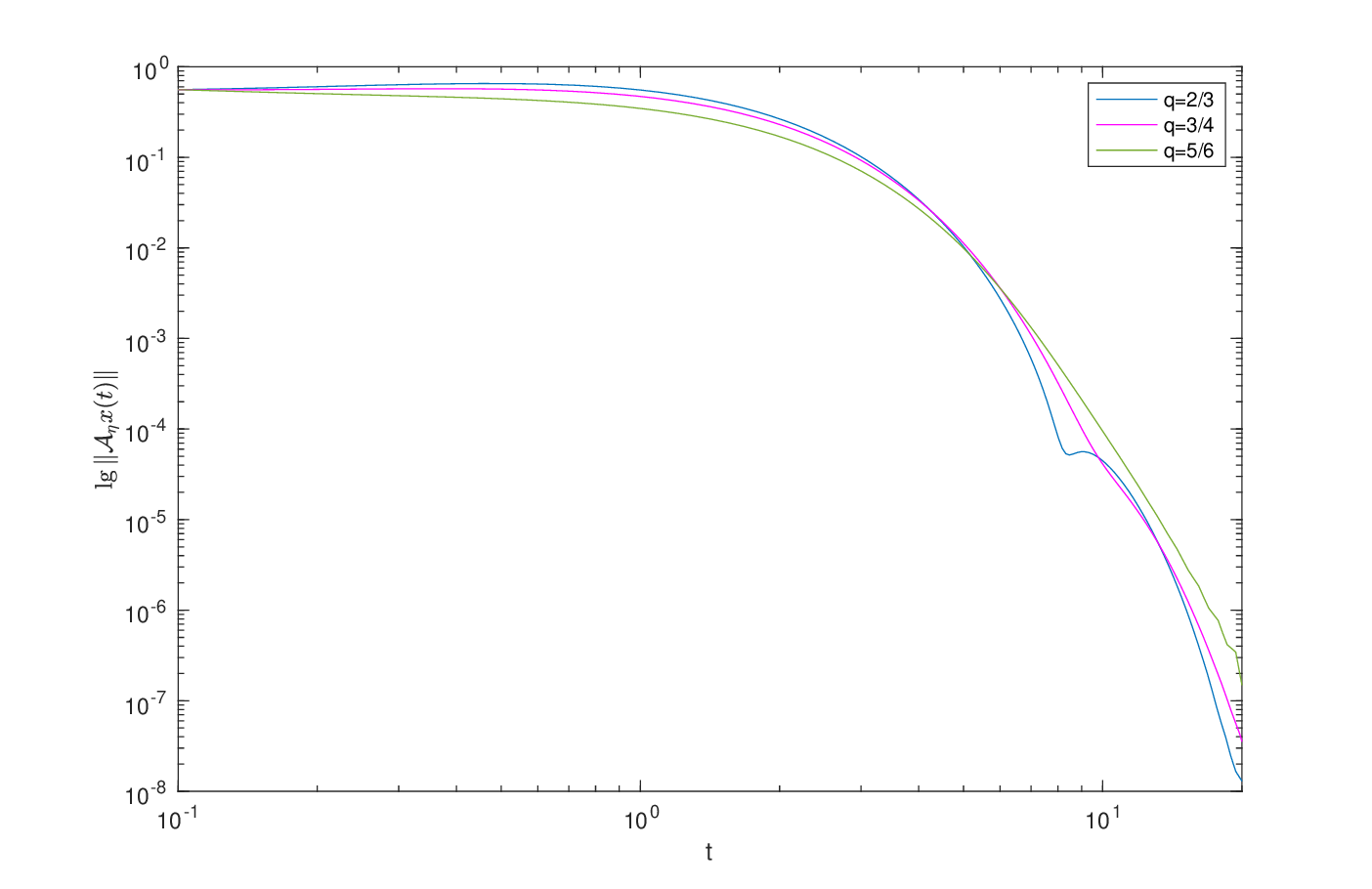}
  \end{minipage}
 }
 \caption{Rescaled iteration errors $\text{ lg } \|x(t)\|$, $\text{ lg } \|\dot{x}(t)\|$, $\text{ lg }\|\mathcal{A}_{\eta}(x(t)\|$ of \eqref{DS} in Example \ref{Ex1}.}
 \label{Figure2}
 \centering
\end{figure*}

\begin{figure*}[h]
 \centering
 {
  \begin{minipage}[t]{0.31\textwidth}
   \centering
   \includegraphics[width=1.8in]{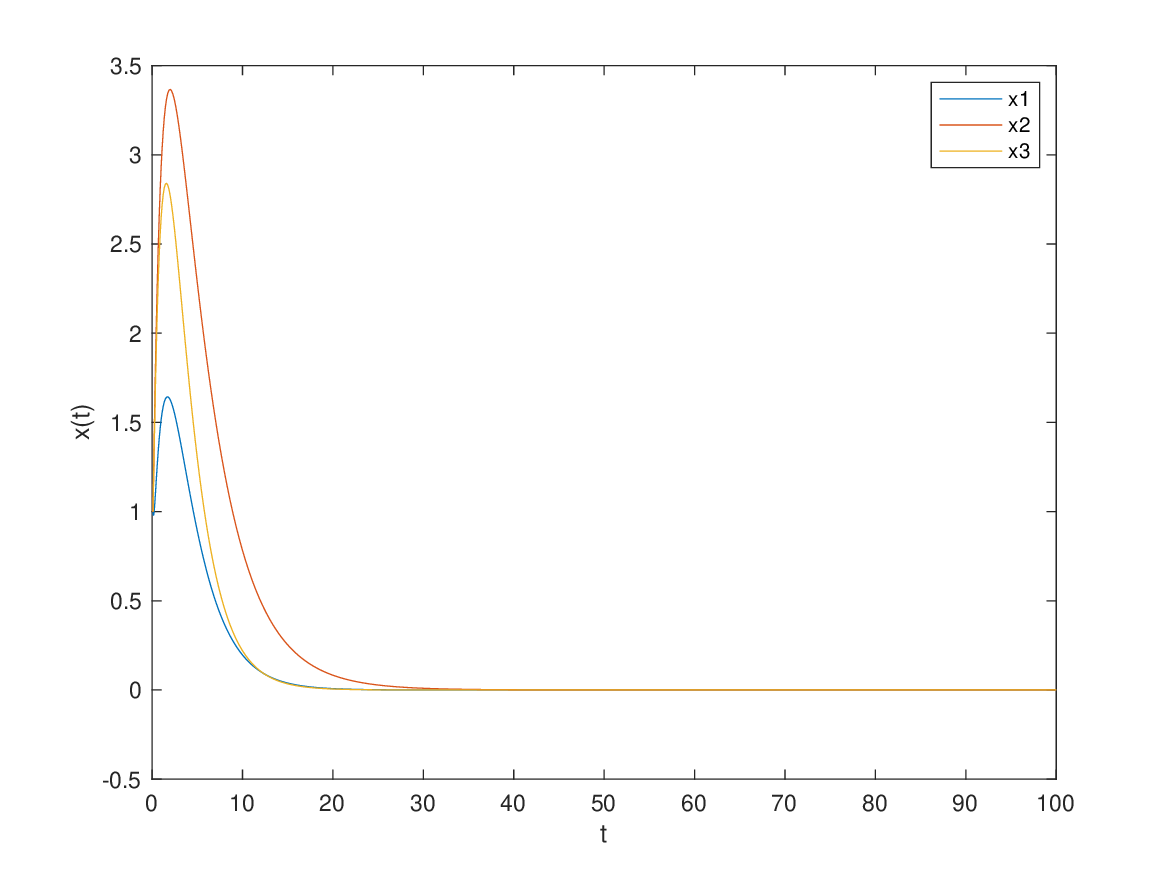}
  \end{minipage}
 }
 {
  \begin{minipage}[t]{0.31\textwidth}
   \centering
   \includegraphics[width=1.8in]{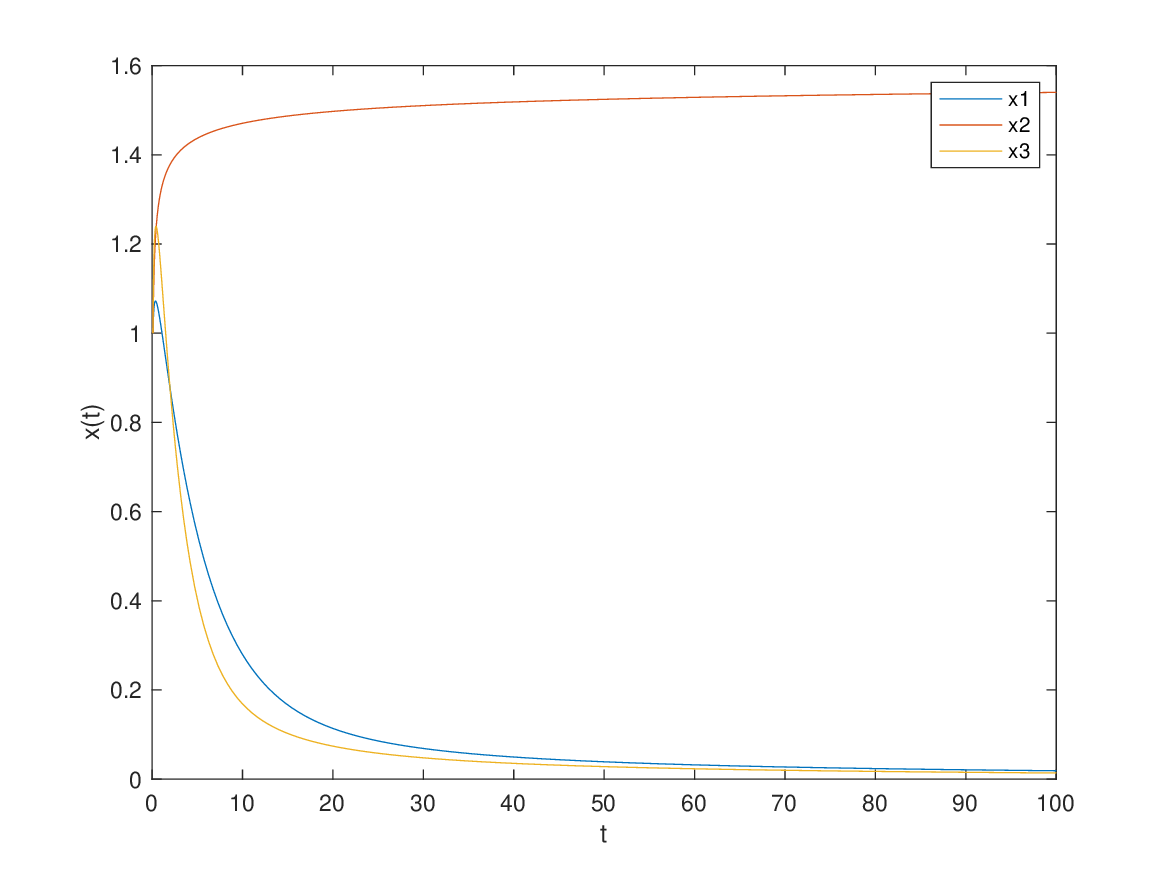}
  \end{minipage}
 }
 \caption{Transient behaviors of $x(t)$ for  systems $(\ref{DS})$ and $(\ref{Tds})$ in Example $\ref{Ex1}$.}
 \label{Figure3}
 \centering
\end{figure*}






\section*{Disclosure statement}
No potential conflict of interest was reported by the authors.

\end{document}